\documentclass[11pt,reqno]{amsart}
\usepackage{amssymb, stmaryrd}
\usepackage{array}
\usepackage{graphicx}
\usepackage{caption, booktabs}
\usepackage[all]{xy}
\usepackage{comment}
\usepackage{cite}
\usepackage[left=2.5cm,right=2.5cm,top=3cm,bottom=3cm]{geometry}
\usepackage{color}
\usepackage{tikz}
\usepackage{tikz-cd}

\makeatletter
\@namedef{subjclassname@1991}{2020 Mathematics Subject Classification}
\makeatother

\newcommand{\g}{\mathfrak{g}}
\newcommand{\h}{\mathfrak{h}}

\newcommand{\p}{\mathfrak{p}}
\newcommand{\uu}{\mathfrak{u}}

\newcommand{\sln}{\mathfrak{sl}}

\newcommand{\ad}{\mathrm{ad}}
\newcommand{\Ad}{\mathrm{Ad}}

\renewcommand{\exp}{\mathrm{exp}}

\numberwithin{equation}{section}

\newtheorem{thm}{Theorem}[section]

\newtheorem{lem}[thm]{Lemma}
\newtheorem{cor}[thm]{Corollary}
\newtheorem{prop}[thm]{Proposition}
\theoremstyle{definition}
\newtheorem{rem}[thm]{Remark}

\newtheorem{ex}[thm]{Example}
\newtheorem{defn}[thm]{Definition}

\newtheorem{obs}[thm]{Observation}

\title[The log symplectic geometry of Poisson slices]{The log symplectic geometry of Poisson slices}

\author[Peter Crooks]{Peter Crooks}
\author[Markus R\"oser]{Markus R\"oser}
\address[Peter Crooks]{Department of Mathematics \\ Northeastern University \\ 360 Huntington Avenue \\ Boston, MA 02115, USA}
\email{p.crooks@northeastern.edu}
\address[Markus R\"oser]{Fachbereich Mathematik\\ Universit\"at Hamburg\\ Bundesstra\ss e 55 (Geomatikum) \\ 20146 Hamburg\\ Germany}
\email{markus.roeser@uni-hamburg.de}

\begin{document}

\date{}

\dedicatory{}

\commby{}

\begin{abstract}
Our paper develops a theory of Poisson slices and a uniform approach to their partial compactifications. The theory in question is loosely comparable to that of symplectic cross-sections in real symplectic geometry.     
\end{abstract}

\maketitle
\tableofcontents

\section{Introduction}
\subsection{Motivation and context}
The Poisson slice construction yields a number of varieties relevant to geometric representation theory and symplectic geometry. One begins with a complex semisimple linear algebraic group with Lie algebra $\g$. Let us also consider a Hamiltonian $G$-variety $X$, i.e. a smooth Poisson variety with a Hamiltonian action of $G$ and moment map $\nu:X\longrightarrow\g$. Each $\mathfrak{sl}_2$-triple $\tau=(\xi,h,\eta)\in\g^{\oplus 3}$ determines a Slodowy slice $$\mathcal{S}_{\tau}:=\xi+\g_{\eta}\subseteq\g,$$ and the preimage $$X_{\tau}:=\nu^{-1}(\mathcal{S}_{\tau})$$ is a Poisson transversal in $X$. The variety $X_{\tau}$ is thereby Poisson, and we call it the \textit{Poisson slice} determined by $X$ and $\tau$. To a certain extent, Poisson slices are complex Poisson-geometric counterparts of symplectic cross-sections \cite{GuilleminPhysics,LermanMeinrenken,GJS,GuilleminLermanSternberg} in real symplectic geometry.

Noteworthy examples of Poisson slices include the product $G\times\mathcal{S}_{\tau}$, a hyperk\"ahler and Hamiltonian $G$-variety studied by Bielawski \cite{Bielawski,BielawskiComplex}, Moore--Tachikawa \cite{MooreTachikawa}, and several others \cite{CrooksBulletin,CrooksRoeser2,CrooksRayan,AbeCrooks,CrooksVP}. A second example is a Coulomb branch \cite{BFN} called the \textit{universal centralizer} $$\mathcal{Z}_{\g}^{\tau}:=\{(g,y)\in G\times\mathcal{S}_\tau:\mathrm{Ad}_{g}(y)=y\},$$
where $\tau$ is a fixed principal $\mathfrak{sl}_2$-triple in $\g$. This hyperk\"ahler variety has received considerable attention in the literature \cite{BFM,BFN,CrooksJGP,Mayrand,Riche,TelemanICM,Teleman}, and it features prominently in B\u{a}libanu's recent paper \cite{BalibanuUniversal}. B\u{a}libanu assumes $G$ to be of adjoint type. She harnesses the geometry of the wonderful compactification $\overline{G}$ and constructs a fibrewise compactification $\overline{\mathcal{Z}_{\g}^{\tau}}\longrightarrow\mathcal{S}_{\tau}$ of $\mathcal{Z}_{\g}^{\tau}\longrightarrow\mathcal{S}_{\tau}$, where the latter map is projection onto the $\mathcal{S}_{\tau}$-factor. She subsequently endows $\overline{\mathcal{Z}_{\g}^{\tau}}$ with a log symplectic structure.  

The preceding discussion gives rise to the following rough questions.
\begin{itemize}
\item Is there a coherent and systematic approach to the partial compactification of Poisson slices that is related to $\overline{G}$ and specializes to yield $\overline{\mathcal{Z}_{\g}^{\tau}}\longrightarrow\mathcal{S}_{\tau}$ as a fibrewise compactification of $\mathcal{Z}_{\g}^{\tau}\longrightarrow\mathcal{S}_{\tau}$?
\item If the previous question has an affirmative answer and $X_{\tau}$ is symplectic, does the partial compactification of $X_{\tau}$ carry a log symplectic structure?      
\end{itemize}
Our inquiry stands to benefit from two observations. One first notes the universal or atomic nature of $G\times\mathcal{S}_{\tau}$ as a Poisson slice, i.e. the existence of a canonical Poisson variety isomorphism
$$X_{\tau}\cong (X\times (G\times\mathcal{S}_{\tau}))\sslash G$$ for each Hamiltonian $G$-variety $X$ and $\mathfrak{sl}_2$-triple $\tau$ in $\g$.  These atomic Poisson slices have counterparts in the theories of symplectic cross-sections \cite{GuilleminPhysics}, symplectic implosion \cite{GJS}, symplectic contraction \cite{Hilgert}, hyperk\"ahler implosion \cite{DancerKirwan,DancerKirwanRoser}, and Kronheimer's hyperk\"ahler quotient with momentum \cite{Kronheimer}. A second observation is that $G\times\mathcal{S}_{\tau}$ sits inside of a larger log symplectic variety $\overline{G\times\mathcal{S}_{\tau}}$ as the unique open dense symplectic leaf; the construction of $\overline{G\times\mathcal{S}_{\tau}}$ assumes $G$ to be of adjoint type and exploits the geometry of $\overline{G}$. 

The preceding considerations motivate us to define
$$\overline{X}_{\tau}:=(X\times (\overline{G\times\mathcal{S}_{\tau}}))\sslash G$$ and conjecture that $\overline{X}_{\tau}$ is the desired partial compactification of $X_{\tau}$. While this naive conjecture needs to be refined and made more precise, it inspires many of the results in our paper.  

\subsection{Summary of Results}

Our paper develops a detailed theory of Poisson slices and addresses the questions posed above. The following is a summary of our results. We work exclusively over $\mathbb{C}$ and take all Poisson varieties to be smooth. We use the Killing form to freely identify $\g^*$ with $\g$, as well as the left trivialization and Killing form to freely identify $T^*G$ with $G\times\mathfrak{g}$.
 
Suppose that $X$ is a Poisson Hamiltonian $G$-variety with moment map $\nu:X\longrightarrow\g$. Let $\tau$ be an $\mathfrak{sl}_2$-triple in $\g$ and consider the Poisson transversal $$X_{\tau}:=\nu^{-1}(\mathcal{S}_{\tau})\subseteq X.$$ The following are some first properties of the Poisson slice $X_{\tau}$. Such properties are well-known in the case of a symplectic variety $X$ (see \cite{Bielawski}).    

\begin{prop}\label{Proposition: Slices}
Let $X$ be a Poisson variety endowed with a Hamiltonian $G$-action and moment map $\nu:X\longrightarrow\g$. Suppose that $\tau=(\xi,h,\eta)$ is an $\mathfrak{sl}_2$-triple in $\g$. The following statements hold.
\begin{itemize}
\item[(i)] The Poisson slice $X_{\tau}$ is transverse to the $G$-orbits in $X$.
\item[(ii)] There are canonical Poisson variety isomorphisms
$$(X\times(G\times\mathcal{S}_{\tau}))\sslash G\cong X_{\tau}\cong X\sslash_\xi U_\tau.$$ The Hamiltonian $G$-variety structure on $G\times\mathcal{S}_{\tau}$ and meaning of the unipotent subgroup $U_{\tau}\subseteq G$ are given in Subsection \ref{Subsection: Poisson slices via Hamiltonian reduction}.
\end{itemize}
\end{prop}   

We also consider some special cases of the Poisson slice construction, including the following well-known result in the symplectic category. 

\begin{obs}\label{Observation}
Let $X$ be a symplectic variety endowed with a Hamiltonian action of $G$ and a moment map $\nu:X\longrightarrow\mathfrak{g}$. Suppose that $\tau$ is an $\mathfrak{sl}_2$-triple in $\mathfrak{g}$. The Poisson structure on $X_{\tau}$ makes it a symplectic subvariety of $X$.
\end{obs}

Now suppose that the above-mentioned Poisson variety $X$ is \textit{log symplectic}  \cite{Gualtieri,Goto,Radko,PymElliptic}, by which the following is meant: $X$ has a unique open dense symplectic leaf, and the degeneracy locus of the Poisson bivector is a reduced normal crossing divisor. We establish the following log symplectic counterpart of Observation \ref{Observation}.

\begin{prop}\label{Proposition: Slight generalization}
Let $X$ be a log symplectic variety endowed with a Hamiltonian $G$-action and moment map $\nu:X\longrightarrow\mathfrak{g}$. Suppose that $\tau$ is any $\mathfrak{sl}_2$-triple in $\mathfrak{g}$. Each irreducible component of $X_{\tau}$ is then a Poisson subvariety of $X_{\tau}$. The resulting Poisson structure on each component makes the component a log symplectic subvariety of $X$.
\end{prop}

Now assume $G$ to be of adjoint type. One may consider the De Concini--Procesi wonderful compactification $\overline{G}$ of $G$ \cite{DeConcini}, along with the divisor $D:=\overline{G}\setminus G$. The data $(\overline{G},D)$ determine a \textit{log cotangent bundle} $T^*\overline{G}(\log(D))$, which is known to have a canonical log symplectic structure. Its unique open dense symplectic leaf is $T^*G$, and the canonical Hamiltonian ($G\times G$)-action on $T^*G$ extends to such an action on $T^*\overline{G}(\log(D))$. The moment maps
$$\rho=(\rho_L,\rho_R):T^*G\longrightarrow\mathfrak{g}\oplus\mathfrak{g}\quad\text{and}\quad\overline{\rho}=(\overline{\rho}_L,\overline{\rho}_R):T^*\overline{G}(\log(D))\longrightarrow\mathfrak{g}\oplus\mathfrak{g}$$ can be written in explicit terms. This leads to the following straightforward observations, whose proofs use Observation \ref{Observation} and Proposition \ref{Proposition: Slight generalization}. To this end, recall that a principal $\mathfrak{sl}_2$-triple is an $\mathfrak{sl}_2$-triple consisting of regular elements.

\begin{obs}
Let $\tau=(\xi,h,\eta)$ be a principal $\mathfrak{sl}_2$-triple in $\g$, and consider the principal $\mathfrak{sl}_2$-triple $(\tau,\tau):=((\xi,\xi),(h,h),(\eta,\eta))$ in $\g\oplus\g$. One then has
$$(T^*G)_{(\tau,\tau)}=\rho^{-1}(\mathcal{S}_{\tau}\times\mathcal{S}_{\tau})=\mathcal{Z}_{\mathfrak{g}}^{\tau},\quad\text{and}\quad (T^*\overline{G}(\log D))_{(\tau,\tau)}=\overline{\rho}^{-1}(\mathcal{S}_{\tau}\times\mathcal{S}_{\tau})=\overline{\mathcal{Z}_{\mathfrak{g}}^{\tau}}.$$ The first Poisson slice is symplectic, while the second is log symplectic. 
\end{obs}

\begin{obs}
Consider the Hamiltonian action of $G=\{e\}\times G\subseteq G\times G$ on $T^*G$. If $\tau$ is an $\mathfrak{sl}_2$-triple in $\g$, then
$$(T^*G)_\tau=\rho_R^{-1}(\mathcal{S}_{\tau})=G\times\mathcal{S}_{\tau}.$$ This Poisson slice is symplectic.
\end{obs}

In light of these observations, it is natural to consider the Poisson slice
$$\overline{G\times\mathcal{S}_{\tau}}:=\overline{\rho}_R^{-1}(\mathcal{S}_{\tau})\subseteq T^*\overline{G}(\log D).$$ One has an inclusion $G\times\mathcal{S}_{\tau}\subseteq \overline{G\times\mathcal{S}_{\tau}}$, while $G\times\mathcal{S}_{\tau}$ and $\overline{G\times\mathcal{S}_{\tau}}$ carry residual Hamiltonian actions of $G=G\times\{e\}\subseteq G\times G$. The respective moment maps are $$\rho_{\tau}:=\rho_L\big\vert_{G\times\mathcal{S}_{\tau}}\quad\text{and}\quad\overline{\rho}_{\tau}:=\overline{\rho}_L\big\vert_{\overline{G\times\mathcal{S}_{\tau}}},$$ and they feature in the following result.

\begin{thm}\label{Theorem: triple}
Let $\tau$ be an $\mathfrak{sl}_2$-triple in $\g$.
\begin{itemize}
\item[(i)] The Poisson slice $\overline{G\times\mathcal{S}_{\tau}}$ is irreducible and log symplectic.
\item[(ii)] The inclusion $G\times\mathcal{S}_{\tau}\longrightarrow\overline{G\times\mathcal{S}_{\tau}}$ is a $G$-equivariant symplectomorphism onto the unique open dense symplectic leaf in $\overline{G\times\mathcal{S}_{\tau}}$.
\item[(iii)] The diagram \begin{equation}\label{Equation: Diagram1}\begin{tikzcd}
G\times\mathcal{S}_{\tau} \arrow{rr} \arrow[swap]{dr}{\rho_\tau}& & \overline{G\times\mathcal{S}_\tau} \arrow{dl}{\overline{\rho}_\tau}\\
& \mathfrak{g} & 
\end{tikzcd}
\end{equation}
commutes.
\item[(iv)] If $\tau$ is a principal $\mathfrak{sl}_2$-triple, then \eqref{Equation: Diagram1} realizes $\overline{\rho}_{\tau}$ as a fibrewise compactification of $\rho_{\tau}$.  
\end{itemize}
\end{thm}

Our paper subsequently discusses the relation of \eqref{Equation: Diagram1} to B\u{a}libanu's fibrewise compactification \begin{equation}\label{Equation: Diagram2}\begin{tikzcd}
\mathcal{Z}_{\g}^{\tau} \arrow{rr} \arrow[swap]{dr}{}& & \overline{\mathcal{Z}_{\g}^{\tau}} \arrow{dl}{}\\
& \mathcal{S}_{\tau} & 
\end{tikzcd}.
\end{equation}

We next study Hamiltonian reductions of the form
$$\overline{X}_{\tau}:=(X\times(\overline{G\times\mathcal{S}_{\tau}}))\sslash G,$$
where $X$ is a Hamiltonian $G$-variety and $\tau$ is an $\mathfrak{sl}_2$-triple in $\g$. The special case $\tau=0$ features prominently in our analysis, and we write $\overline{X}$ for $\overline{X}_{\tau}$ if $\tau=0$. This amounts to setting
$$\overline{X}:=(X\times T^*\overline{G}(\log D))\sslash G,$$ with $G$ acting as $G=G\times\{e\}\subseteq G\times G$ on $T^*\overline{G}(\log D)$.

One has the preliminary issue of whether the quotients $\overline{X}_{\tau}$ and $\overline{X}$ exist. The following result provides some sufficient conditions.

\begin{lem}
Let $X$ be a Hamiltonian $G$-variety.
\begin{itemize}
\item[(i)] If $X$ is a principal $G$-bundle, then $\overline{X}$ exists as a geometric quotient.
\item[(ii)] If $\overline{X}$ exists as a geometric quotient, then $\overline{X}_{\tau}$ exists as a geometric quotient for all $\mathfrak{sl}_2$-triples $\tau$ in $\g$.
\item[(iii)] If $X=T^*Y$ for an irreducible smooth principal $G$-bundle $Y$, then $\overline{X}_{\tau}$ exists as a geometric quotient for all $\mathfrak{sl}_2$-triples $\tau$ in $\g$. 
\end{itemize}
\end{lem} 

The variety $\overline{X}_{\tau}$ enjoys certain Poisson-geometric features. A first step in this direction is to set
$$\overline{X}_{\tau}^{\circ}:=(X\times(\overline{G\times\mathcal{S}_{\tau}}))^{\circ}\sslash G,$$ where $(X\times\overline{G\times\mathcal{S}_{\tau}})^{\circ}$ is the open set of points in $(X\times(\overline{G\times\mathcal{S}_{\tau}}))^{\circ}$ with trivial $G$-stabilizers. The variety $\overline{X}_{\tau}^{\circ}$ exists as a geometric quotient if $\overline{X}$ exists as a geometric quotient, in which case one has inclusions
$$X_{\tau}\subseteq\overline{X}_{\tau}^{\circ}\subseteq\overline{X}_{\tau}$$ 
\begin{thm}\label{Theorem: many}
Let $X$ be a Hamiltonian $G$-variety, and suppose that $\tau$ is an $\mathfrak{sl}_2$-triple in $\g$. Assume that $\overline{X}_{\tau}$ exists as a geometric quotient.
\begin{itemize}
\item[(i)] The coordinate ring $\mathbb{C}[\overline{X}_{\tau}]$ carries a natural Poisson bracket for which restriction $\mathbb{C}[\overline{X}_{\tau}]\longrightarrow\mathbb{C}[X_{\tau}]$ is a Poisson algebra morphism.
\item[(ii)] The variety $\overline{X}_{\tau}^{\circ}$ is smooth and Poisson, and it contains $X_{\tau}$ as an open Poisson subvariety.
\item[(iii)] If $X$ is symplectic, then each irreducible component of $\overline{X}_{\tau}^{\circ}$ is log symplectic.
\item[(iv)] If $X$ is symplectic and $X_{\tau}$ is irreducible, then $X_{\tau}$ is the open dense symplectic leaf in a unique irreducible component of $\overline{X}_{\tau}^{\circ}$.    
\end{itemize}
\end{thm}
 
Our final main result addresses the extent to which $\overline{X}_{\tau}$ partially compactifies $X_{\tau}$. We begin by assuming that both $\overline{X}_{\tau}$ and $X/G$ exist as geometric quotients. This allows us to construct canonical maps
$$\pi_{\tau}:X_{\tau}\longrightarrow X/G\quad\text{and}\quad\overline{\pi}_{\tau}:\overline{X}_{\tau}\longrightarrow X/G.$$ It is then straightforward to deduce that  
\begin{equation}\begin{tikzcd}\label{Equation: Partial compactification}
X_{\tau} \arrow{rr}{} \arrow[swap]{dr}{\pi_{\tau}}& & \overline{X}_{\tau} \arrow{dl}{\overline{\pi}_{\tau}}\\
& X/G & 
\end{tikzcd}\end{equation}
commutes, where the horizontal arrow is inclusion. This leads to the following theorem.

\begin{thm}\label{Theorem: Compactification theorem}
Let $X$ be a Hamiltonian $G$-variety, and suppose that $\tau$ is a principal $\mathfrak{sl}_2$-triple in $\g$. If $\overline{X}_{\tau}$ and $X/G$ exist as geometric quotients, then \eqref{Equation: Partial compactification} realizes $\overline{\pi}_{\tau}$ as a fibrewise compactification of $\pi_{\tau}$.
\end{thm}

In the case of a principal $\mathfrak{sl}_2$-triple $\tau$, we realize the fibrewise compactifications \eqref{Equation: Diagram1} and \eqref{Equation: Diagram2} as special instances of Theorem \ref{Theorem: Compactification theorem}.  

\subsection{Organization}
In Section \ref{Section: Preliminaries}, we introduce the concepts from Lie theory and Poisson geometry that form the foundation for our work. Section \ref{Section: Poisson slices} details the theory of Poisson slices and provides  complete proofs of Propositions \ref{Proposition: Slices} and \ref{Proposition: Slight generalization}. Section \ref{Section: Xbar} subsequently considers the Poisson slice enlargements $\overline{X}_{\tau}$ mentioned above, and it contains the proofs of Theorems \ref{Theorem: triple}, \ref{Theorem: many} and \ref{Theorem: Compactification theorem}. We conclude with Section \ref{Section: Examples}, which is devoted to illustrative examples. A list of recurring notation appears after Section \ref{Section: Examples}.

\subsection*{Acknowledgements}
The first author gratefully acknowledges the Natural Sciences and Engineering Research Council of Canada for support [PDF--516638].

\section{Preliminaries}
\label{Section: Preliminaries}

This section provides some of the notation, conventions, and basic results used throughout our paper. 

\subsection{Fundamental conventions}
We work exclusively over $\mathbb{C}$ and understand all group actions as being left group actions. We also write $\mathcal{O}_X$ for the structure sheaf of an algebraic variety $X$, as well as $\mathbb{C}[X]$ for the coordinate ring $\mathcal{O}_X(X)$. The dimension of $X$ is understood to be the supremum of the dimensions of the irreducible components. We understand $X$ to be smooth if $\dim(T_xX)=\dim(X)$ for all $x\in X$. Note that this convention forces a smooth variety to be pure-dimensional.
   
\subsection{Quotients of $K$-varieties}
Let $K$ be a linear algebraic group. We adopt the term \textit{$K$-variety} in reference to a variety $X$ endowed with an algebraic $K$-action.

\begin{defn}\label{Definition: GoodQuotient}
	Suppose that $X$ is a $K$-variety. A variety morphism $\pi:X\longrightarrow Y$ is called a \textit{categorical quotient} of $X$ if the following conditions are satisfied:
	\begin{itemize}
		\item[(i)] $\pi$ is $K$-invariant;
		\item[(ii)] if $\theta:X\longrightarrow Z$ is a $K$-invariant variety morphism, then there exists a unique morphism $\varphi:Y\longrightarrow Z$ for which
		$$\begin{tikzcd}
		& X \arrow{dl}[swap]{\pi} \arrow{dr}{\theta} \\
		Y \arrow{rr}[swap]{\varphi} && Z
		\end{tikzcd}$$
		commutes.
	\end{itemize}
\end{defn}

\begin{defn}
	Suppose that $X$ is a $K$-variety. A variety morphism $\pi:X\longrightarrow Y$ is called a \textit{good quotient} of $X$ if the following conditions are satisfied:
	\begin{itemize}
		\item[(i)] $\pi$ is surjective, affine, and $K$-invariant;
		\item[(ii)] if $U\subseteq Y$ is open, then the comorphism $\pi^*:\mathcal{O}_Y(U)\longrightarrow\mathcal{O}_X(\pi^{-1}(U))$ is an isomorphism onto $\mathcal{O}_X(\pi^{-1}(U))^K$;
		\item[(iii)] if $Z\subseteq X$ is closed and $K$-invariant, then $\pi(Z)$ is closed in $Y$;
		\item[(iv)] if $Z_1,Z_2\subseteq X$ are closed, $K$-invariant, and disjoint, then $\pi(Z_1)$ and $\pi(Z_2)$ are disjoint.	
	\end{itemize}
	One calls $\pi:X\longrightarrow Y$ a \textit{geometric quotient} of $X$ if $\pi$ is a good quotient and $\pi^{-1}(y)$ is a $K$-orbit for each $y\in Y$. 	
\end{defn}

Let $X$ be a $K$-variety admitting a geometric quotient $\pi:X\longrightarrow Y$, and write $X/K$ for the set of $K$-orbits in $X$. One then has a canonical bijection $Y\cong X/K$, through which $X/K$ inherits a variety structure. Any two geometric quotients $\pi:X\longrightarrow Y$ and $\pi':X\longrightarrow Y'$ induce the same variety structure on $X/K$, and this structure makes the set-theoretic quotient map $X\longrightarrow X/K$ a geometric quotient. With this in mind, we shall sometimes write ``$X/K$ exists" or ``the geometric quotient $X\longrightarrow X/K$ exists" to mean that $X$ admits a geometric quotient.

\begin{lem}\label{Lemma: Affine}
Assume that $K$ is connected and reductive. Suppose that $X$ is an affine $K$-variety, and let $\pi:X\longrightarrow Y$ be a variety morphism. If $\pi^{-1}(y)$ is a $K$-orbit for each $y\in Y$, then $\pi$ is a geometric quotient of $X$.
\end{lem}

\begin{proof}
One knows that $X$ admits a good quotient $\theta:X\longrightarrow Z$ (e.g. \cite[Theorem 1.4.2.4]{Schmitt}), and that this quotient is also categorical (e.g. \cite[Lemma 1.4.1.1]{Schmitt}). It follows that \begin{equation}\label{Equation: Diagram}\begin{tikzcd}
	& X \arrow{dl}[swap]{\theta} \arrow{dr}{\pi} \\
	Z \arrow{rr}[swap]{\varphi} && Y
	\end{tikzcd}\end{equation}
	commutes for some morphism $\varphi:Z\longrightarrow Y$. The surjectivity of $\pi$ then forces $\varphi$ to be surjective. 
	
	We claim that $\varphi$ is an isomorphism, i.e. that $\varphi$ is also injective. To this end, let $z_1,z_2\in Z$ be such that $\varphi(z_1)=\varphi(z_2)$. Choose $x_1,x_2\in X$ satisfying $\theta(x_1)=z_1$ and $\theta(x_2)=z_2$. One then has $\pi(x_1)=\pi(x_2)$, so that $x_1$ and $x_2$ belong to the same $K$-orbit in $X$. The $K$-invariance of $\theta$ now implies that $\theta(x_1)=z_1$ and $\theta(x_2)=z_2$ must coincide. We conclude that $\varphi$ is injective, implying that $\varphi$ is indeed an isomorphism. This combines with the commutativity of \eqref{Equation: Diagram} to tell us that $\pi$ is a good quotient. Since $\pi^{-1}(y)$ is a $K$-orbit for all $y\in Y$, one deduces that $\pi$ is a geometric quotient. 
\end{proof}

\begin{prop}\label{Proposition: Restriction of quotient}
Assume that $K$ is connected and reductive. Suppose that $X$ is a $K$-variety admitting a geometric quotient $\pi:X\longrightarrow Y$. If $Z\subseteq X$ is a $K$-invariant locally closed subvariety, then $\pi(Z)$ is locally closed in $Y$ and 
$$\pi\big\vert_Z:Z\longrightarrow\pi(Z)$$ is a geometric quotient of $Z$.
\end{prop}

\begin{proof}
We begin by showing $\pi(Z)$ to be locally closed in $Y$. To this end, note that the closure $\overline{Z}\subseteq X$ is $K$-invariant. It follows that $\pi(\overline{Z})$ is a closed subvariety of $Y$. An application of \cite[Lemma 25.3.2]{Tauvel} now forces $$\pi\big\vert_{\overline{Z}}:\overline{Z}\longrightarrow\pi(\overline{Z})$$ to be an open map. Since $Z$ is open in $\overline{Z}$, we conclude that $\pi(Z)$ is open in the closed subvariety $\pi(\overline{Z})\subseteq Y$. This implies that $\pi(Z)$ is locally closed in $Y$.
	
	Now observe that $\pi\big\vert_Z:Z\longrightarrow\pi(Z)$ is the base change of $\pi$ under the locally closed immersion $\pi(Z)\hookrightarrow Y$. Since affine morphisms are stable under base change (e.g. \cite[Exercise 4.E.9(2)]{Patil}), it follows that $\pi\big\vert_{Z}$ is affine. We may therefore cover $\pi(Z)$ with affine open subsets $\{V_{\alpha}\}_{\alpha\in\mathcal{A}}$ in such a way that $$\{U_{\alpha}:=(\pi\big\vert_Z)^{-1}(V_{\alpha})\}_{\alpha\in\mathcal{A}}$$ is a cover of $Z$ by affine open subsets. By virtue of \cite[Exercise 1.4.1.2 ii)]{Schmitt}, we are reduced to verifying that each morphism
	$$\pi\big\vert_{U_{\alpha}}:U_{\alpha}\longrightarrow V_{\alpha}$$ is a geometric quotient. This reduction combines with Lemma \ref{Lemma: Affine} to complete the proof. 
\end{proof}

We will need the following algebro-geometric notion of a principal bundle appearing in \cite[Definition 2.3.1]{BrionHandbook}. 

\begin{defn}
Suppose that $X$ is a $K$-variety. A $K$-invariant variety morphism $\pi: X\longrightarrow Y$  is called a \emph{principal $K$-bundle} if the following conditions hold:
\begin{itemize}
\item[(i)] $\pi$ is faithfully flat, i.e. flat and surjective; 
\item[(ii)] the natural map $$\sigma: K\times X\longrightarrow X\times_{Y}X,\quad \sigma(k,x) = (x,k\cdot x)$$ is an isomorphism.
\end{itemize}
\end{defn}

A principal $K$-bundle is necessarily a geometric quotient (e.g. by \cite[Proposition 2.3.3]{BrionHandbook}). We understand ``$X$ is a principal $K$-bundle" as meaning that $X$ admits a geometric quotient $\pi:X\longrightarrow X/K$, and that $\pi$ is a principal $K$-bundle.

\begin{lem}\label{Lemma: FreeActionGbundle}
Let $X$ be a smooth $K$-variety with a free $K$-action and a good quotient $\pi\colon X\longrightarrow Y$. The map $\pi$ is then a principal $K$-bundle.
\end{lem}

\begin{proof}
Since the $K$-action is free, all $K$-orbits are closed in $X$. It now follows from Definition \ref{Definition: GoodQuotient}(iv) that $\pi$ separates $K$-orbits. Each fibre of $\pi$ is therefore a single $K$-orbit. This combines with the smoothness of $X$ and the freeness of the $K$-action to imply that $\pi$ is flat.  

It remains only to prove that 
$$\sigma: K\times X\longrightarrow X\times_{Y}X,\quad \sigma(k,x) = (x,k\cdot x)$$
is an isomorphism. This follows immediately from the freeness of the $K$-action and the fact that the fibres of $\pi$ are the $K$-orbits in $X$.
\end{proof}

\subsection{Poisson varieties}
Let $X$ be a smooth variety. Suppose that $P$ is a global section of $\Lambda^2(TX)$, and consider the bracket operation defined by
$$\{f_1,f_2\}:=P(df_1\wedge df_2)\in\mathcal{O}_X$$ for all $f_1,f_2\in\mathcal{O}_X$. One calls $P$ a \textit{Poisson bivector} if this bracket renders $\mathcal{O}_X$ a sheaf of Poisson algebras. We use the term \textit{Poisson variety} in reference to a smooth variety $X$ equipped with a Poisson bivector $P$. In this case, $\{\cdot,\cdot\}$ is called the \textit{Poisson bracket}. Let us also recall that a variety morphism $\phi: X_1\longrightarrow X_2$ between Poisson varieties $(X_1,P_1)$ and $(X_2,P_2)$ is called a \textit{Poisson morphism} if 
$$d\phi(P_1(\phi^*\alpha))= P_2(\alpha)$$
for all one-forms $\alpha$ defined on any open subset of $X_2$. Our convention is to have $(X_1\times X_2,P_1\oplus (-P_2))$ be the Poisson variety product of $(X_1,P_1)$ and $(X_2,P_2)$.

Let $(X,P)$ be a Poisson variety. Contracting the bivector with cotangent vectors allows one to view $P$ as a bundle morphism
$$P:T^*X\longrightarrow TX,$$ whose image is a holomorphic distribution on $X$. One refers to the maximal integral submanifolds of this distribution as the \textit{symplectic leaves} of $X$. The symplectic form $\omega_L$ on a symplectic leaf $L\subseteq X$ is constructed as follows. One defines the \textit{Hamiltonian vector field} of a locally defined function $f\in\mathcal{O}_X$ by
\begin{equation}\label{Equation: HamiltonianVectorField}
H_f := -P(df).
\end{equation} This gives rise to the tangent space description
$$T_xL=\{(H_f)_x:f\in\mathcal{O}_X\}$$
for all $x\in L$, and one has
$$(\omega_L)_x((H_{f_1})_x,(H_{f_2})_x)=\{f_1,f_2\}(x)$$ for all $x\in L$ and $f_1,f_2\in\mathcal{O}_X$ defined near $x$.

We conclude by discussing \textit{log symplectic varieties}, which have received considerable attention in recent years  (e.g. \cite{BalibanuUniversal,GuilleminIMRN,Goto,Radko,GMP,
GuilleminAdv,Cavalcanti,PymElliptic,Gualtieri,MarkusTorres1,
MarkusTorres2,GualtieriLi,GualtieriPym,PymSchedler}). To this end, one calls a Poisson variety $(X,P)$ \textit{log symplectic} if  the following conditions hold:
\begin{itemize}
\item[(i)] $(X,P)$ has a unique open dense symplectic leaf $X_{0}\subseteq X$;
\item[(ii)]  the vanishing locus of $P^n$ is a reduced normal crossing divisor $D\subseteq X$, where $2n=\dim(X_{0})$ and $P^n\in H^0(X,\Lambda^{2n}(TX))$ is the top exterior power of $P$.
\end{itemize}
In this case, we call $D$ the \textit{divisor} of $(X,P)$.  One immediate observation is that $D=X\setminus X_{0}$.

\begin{rem}
Since symplectic leaves are connected, Condition (i) implies that log symplectic varieties are irreducible.
\end{rem}
 
\subsection{Hamiltonian reduction}\label{Subsection: Hamiltonian reduction}

We now review the salient aspects of Hamiltonian actions in the Poisson category. To this end, let $K$ be a linear algebraic group with Lie algebra $\mathfrak{k}$.
Let $(X,P)$ be a Poisson variety, and assume that $X$ is also a $K$-variety. Each $y\in\mathfrak{k}$ then determines a \textit{fundamental vector field} $V_y$ on $X$ via
$$(V_y)_x=\frac{d}{dt}\bigg\vert_{t=0}(\exp(ty)\cdot x)\in T_xX$$
for all $x\in X$. The $K$-action on $X$ is called \textit{Hamiltonian} if $P$ is $K$-invariant and there exists a $K$-equivariant morphism $\nu:X\longrightarrow\mathfrak{k}^*$ satisfying the following condition:
\begin{equation}\label{Equation: MomentCondition}
H_{\nu^y}=-V_{y}
\end{equation}
 for all $y\in\mathfrak{k}$, where $\nu^y\in\mathbb{C}[X]$ is defined by
\begin{equation}\label{Equation: Defnuy}
\nu^y(x)= \nu(x)(y),\quad x\in X.
\end{equation}
One then refers to $\nu$ as a moment map and calls $(X,P,\nu)$ a \textit{Hamiltonian $K$-variety}. The moment map $\nu$ is known to be a Poisson morphism with respect to the Lie--Poisson structure on $\mathfrak{k}^*$ (e.g. \cite[Proposition 7.1]{Weinstein}).

We now briefly recall the process of Hamiltonian reduction for a Hamiltonian $K$-variety $(X,P,\nu)$. One begins by observing that $\nu^{-1}(0)$ is a $K$-invariant closed subvariety of $X$. Let us assume $X\sslash K$ exists, by which we mean that the geometric quotient
\begin{equation}\label{Equation: Geometric}\pi:\nu^{-1}(0)\longrightarrow\nu^{-1}(0)/K\end{equation} exists. Write
$$X\sslash K:=\nu^{-1}(0)/K,$$ and note that the comorphism $\pi^*:\mathbb{C}[X\sslash K]\longrightarrow\mathbb{C}[\nu^{-1}(0)]$ induces an algebra isomorphism
\begin{equation}\label{Equation: Coord}\mathbb{C}[X\sslash K]\overset{\cong}\longrightarrow\mathbb{C}[\nu^{-1}(0)]^K.\end{equation}
At the same time, the canonical surjection $\mathbb{C}[X]\longrightarrow\mathbb{C}[\nu^{-1}(0)]$ restricts to a surjection
\begin{equation}\label{Equation: Surjection}\mathbb{C}[X]^K\longrightarrow\mathbb{C}[\nu^{-1}(0)]^K\end{equation} if $K$ is connected and reductive. One also knows that $\mathbb{C}[X]^K$ is a Poisson subalgebra of $\mathbb{C}[X]$, and that the kernel of \eqref{Equation: Surjection} is a Poisson ideal $I\subseteq\mathbb{C}[X]^K$. It follows that $\mathbb{C}[\nu^{-1}(0)]^K$ inherits the structure of a Poisson algebra. One may therefore endow $\mathbb{C}[X\sslash K]$ with the unique Poisson bracket for which \eqref{Equation: Coord} is an isomorphism of Poisson algebras. We refer to the data of the variety $X\sslash K$ and the Poisson algebra $\mathbb{C}[X\sslash K]$ as the \textit{Hamiltonian reduction} of $(X,P,\nu)$ if \eqref{Equation: Geometric} exists and $K$ is connected and reductive.

The Hamiltonian reduction process will yield a richer geometric object in the presence of certain assumptions about the $K$-action on $X$. To this end, let $K$ be a linear algebraic group and suppose that $(X,P,\nu)$ is a Hamiltonian $K$-variety. Assume that the geometric quotient \eqref{Equation: Geometric} exists, and that $K$ acts freely on $\nu^{-1}(0)$. The closed subvariety $\nu^{-1}(0)\subseteq X$ is then smooth, and one deduces that $X\sslash K$ is also smooth. One may also define a Poisson bivector $P_{X\sslash K}$ on $X\sslash K$ as follows. Suppose that $x\in \nu^{-1}(0)$ and let 
$$d\pi_x^*:T_{\pi(x)}^*(X\sslash K)\longrightarrow T_x^*(\nu^{-1}(0))$$ 
be the dual of the differential $d\pi_x:T_x(\nu^{-1}(0))\longrightarrow T_{\pi(x)}(X\sslash K)$. Set
$$P_{\pi(x)}(\alpha):=d\pi_x(P_x(\tilde{\alpha}))$$ for all $\alpha\in T_{\pi(x)}^*(X\sslash K)$, where $\tilde{\alpha}\in T_x^*X$ is any element that annihilates $T_x(Kx)$ and coincides with $d\pi_x^*(\alpha)$ on $T_x(\nu^{-1}(0))$. The bivector $P_{X\sslash K}$ renders $\mathcal{O}_{X\sslash K}$ a sheaf of Poisson algebras, recovering the above-described Poisson bracket on $\mathbb{C}[X\sslash K]$. We call the Poisson variety $(X\sslash_\zeta K ,P_{X\sslash_\zeta K})$ the \textit{Hamiltonian reduction} of $(X,P,\nu)$ \textit{at level $\zeta$}, provided that \eqref{Equation: Geometric} exists and $K$ is known to act freely on $\nu^{-1}(\zeta)$.

The preceding construction generalizes to allow for Hamiltonian reduction at an arbitrary level $\zeta\in\mathfrak{k}^*$. To this end, let $K_{\zeta}$ denote the $K$-stabilizer of $\zeta$ with respect to the coadjoint action. One simply sets $$X\sslash_\zeta K := \nu^{-1}(\zeta)/K_\zeta$$ if the right-hand side exists as a geometric quotient. The definitions of the Poisson bracket on $\mathbb{C}[X\sslash_{\zeta} K]$ and Poisson bivector $P_{X\sslash_\zeta K}$ are analogous to their counterparts above. 

\subsection{Lie-theoretic conventions}\label{Subsection: Lie theoretic conventions}
Let $G$ be a connected semisimple linear algebraic group with Lie algebra $\g$. Note that $\g$ is a $G$-module via the adjoint representation 
$$\mathrm{Ad}:G\longrightarrow\operatorname{GL}(\g),\quad g\longrightarrow\mathrm{Ad}_g,$$ and a $\g$-module via the other adjoint representation
$$\mathrm{ad}:\g\longrightarrow\mathfrak{gl}(\g),\quad y\longrightarrow\mathrm{ad}_y=[y,\cdot].$$ 
One obtains an induced action of $G$ on the coordinate ring $\mathbb{C}[\g]=\mathrm{Sym}(\g^*)$, and we write $\mathbb{C}[\g]^G\subseteq\mathbb{C}[\g]$ for the subalgebra of all functions fixed by $G$. The inclusion $\mathbb{C}[\g]^G\subseteq\mathbb{C}[\g]$ then determines a morphism of affine varieties
$$\chi:\g\longrightarrow\mathrm{Spec}(\mathbb{C}[\g]^G),$$ often called the \textit{adjoint quotient}.

Define the centralizer subalgebra
$$\g_y:=\{z\in\g:[y,z]=0\}\subseteq\g$$ for each $y\in\g$. An element $y\in\g$ is called \textit{regular} if the dimension of $\g_y$ coincides with the rank of $\g$. The set of all regular elements is a $G$-invariant open dense subvariety of $\g$ that we denote by $\g^{\text{r}}$. 

Recall that $(\xi,h,\eta)\in\g^{\oplus 3}$ is an $\mathfrak{sl}_2$-triple if the identities
$$[\xi,\eta]=h,\quad [h,\xi]=2\xi,\quad\text{and}\quad [h,\eta]=-2\eta$$ 
hold in $\g$, and that the associated \textit{Slodowy slice} is defined by
$$\mathcal{S}_{\tau}:=\xi+\g_{\eta}\subseteq\g.$$ 
Now assume that $\tau$ is a \textit{principal} $\mathfrak{sl}_2$-triple, i.e. an $\mathfrak{sl}_2$-triple for which $\xi,h,\eta\in\mathfrak{g}^{\text{r}}$. The slice $\mathcal{S}_{\tau}$ then lies in $\g^{\text{r}}$ and is a fundamental domain for the $G$-action on $\g^{\text{r}}$ \cite[Theorem 8]{KostantLie}. This slice is also known to be a section of the adjoint quotient, meaning that the restriction
$$\chi\big\vert_{\mathcal{S}_{\tau}}:\mathcal{S}_{\tau}\longrightarrow\mathrm{Spec}(\mathbb{C}[\g]^G)$$ 
is a variety isomorphism \cite[Theorem 7]{KostantLie}. Let us write 
$$y_{\tau}:=(\chi\big\vert_{\mathcal{S}_{\tau}})^{-1}(\chi(y))\in\mathcal{S}_{\tau}$$ 
for each $y\in\g$. In other words, $y_{\tau}$ is the unique point at which $\mathcal{S}_{\tau}$ meets $\chi^{-1}(\chi(y))$.

Let $\langle\cdot,\cdot\rangle:\g\otimes_{\mathbb{C}}\g\longrightarrow\mathbb{C}$ denote the Killing form on $\g$. This bilinear form is non-degenerate and $G$-invariant, i.e. the map
\begin{equation}\label{Equation: Identification}\g\longrightarrow\g^*,\quad y\longrightarrow\langle y,\cdot\rangle
\end{equation} 
is a $G$-module isomorphism. The canonical Poisson structure on $\g^*$ thereby corresponds to a Poisson structure on $\g$, determined by the following condition:
$$\{f_1,f_2\}(y)=\langle y,[(df_1)_y,(df_2)_y]\rangle$$
for all $f_1,f_2\in\mathbb{C}[\g]$ and $y\in\g$, where the right-hand side uses \eqref{Equation: Identification} to regard $(df_1)_y,(df_2)_y\in\g^*$ as elements of $\g$. 
By means of \eqref{Equation: Identification}, we shall make no further distinction between $\g$ and $\g^*$. One also has the ($G\times G$)-module isomorphism
	$$\g\oplus\g\longrightarrow(\g\oplus\g)^*,\quad (x_1,x_2)\longrightarrow (\langle x_1,\cdot\rangle,-\langle x_2,\cdot\rangle),$$
	through which we shall identify $\g\oplus\g$ with $(\g\oplus\g)^*$.

\subsection{The wonderful compactification}
In this subsection, we assume that $G$ is the adjoint group of $\g$. 
Let $n=\dim\g$ and write $\mathrm{Gr}(n,\g\oplus\g)$ for the Grassmannian of all $n$-dimensional subspaces in $\g\oplus\g$. Note that $G\times G$ acts on $\mathrm{Gr}(n,\g\oplus\g)$ by
$$(g_1,g_2)\cdot\gamma:=\{(\mathrm{Ad}_{g_1}(y_1),\mathrm{Ad}_{g_2}(y_2)):(y_1,y_2)\in\gamma\},$$
and on $G$ itself by
$$(g_1,g_2)\cdot h:=g_1hg_2^{-1}.$$
Let $\mathfrak{g}_{\Delta}\subseteq\g\oplus\g$ denote the diagonally embedded copy of $\g$ in $\g\oplus\g$, and consider the ($G\times G$)-equivariant locally closed immersion
\begin{equation}\label{Equation: Loc}\varphi:G\longrightarrow\mathrm{Gr}(n,\g\oplus\g),\quad g\longrightarrow (g,e)\cdot\mathfrak{g}_{\Delta}.
\end{equation} 
We thereby view $G$ as a subvariety of $\mathrm{Gr}(n,\g\oplus\g)$ and write $\overline{G}$ for its closure in $\mathrm{Gr}(n,\g\oplus\g)$. The closed subvariety $\overline{G}$ is ($G\times G$)-invariant, smooth, and called the \textit{wonderful compactification} of $G$ \cite{DeConcini}. The complement $D:=\overline{G}\setminus G$ is known to be a normal crossing divisor in $\overline{G}$.

The pair $(G,D)$ determines a so-called \textit{log cotangent bundle} $T^*\overline{G}(\log D)\longrightarrow\overline{G}$. One may realize this vector bundle as the pullback of the tautological bundle $\mathcal{T}\longrightarrow\mathrm{Gr}(n,\g\oplus\g)$ along the inclusion $\overline{G}\hookrightarrow\mathrm{Gr}(n,\g\oplus\g)$. This amounts to setting
$$T^*\overline{G}(\log D):=\{(\gamma,(y_1,y_2))\in\overline{G}\times(\g\oplus\g):(y_1,y_2)\in\gamma\}$$ 
and defining the bundle projection to be
$$T^*\overline{G}(\log D)\longrightarrow\overline{G},\quad (\gamma,(y_1,y_2))\longrightarrow\gamma.$$ 
The action of ($G\times G$) on $\overline{G}$ then lifts to the following ($G\times G$)-action on $T^*\overline{G}(\log D)$:
\begin{equation}\label{Equation: Action on log cotangent}(g_1,g_2)\cdot (\gamma,(y_1,y_2)):=((g_1,g_2)\cdot\gamma,(\mathrm{Ad}_{g_1}(y_1),\mathrm{Ad}_{g_2}(y_2))).
\end{equation}

\subsection{Poisson geometry on $T^*G$ and $T^*\overline{G}(\log D)$}\label{Subsection: MomentLogCotangent}
Let all objects and notation be as set in \ref{Subsection: Lie theoretic conventions}. Note that the left trivialization and Killing form combine to yield a variety isomorphism 
$$T^*G\cong G\times\mathfrak{g}.$$  We shall thereby make no further distinction between $T^*G$ and $G\times\g$. The canonical symplectic form $\omega$ on $T^*G$ is then defined as follows on each tangent space $T_{(g,x)}(G\times\mathfrak{g})=T_gG\oplus\mathfrak{g}$:
$$\omega_{(g,x)}\bigg(\big((dL_g)_e(y_1),z_1\big),\big((dL_g)_e(y_2),z_2\big)\bigg)=\langle y_1,z_2\rangle-\langle y_2,z_1\rangle+\langle x,[y_1,y_2]\rangle$$
for all $y_1,y_2,z_1,z_2\in\mathfrak{g}$, where $L_g:G\longrightarrow G$ denotes left translation by $g$ and $(dL_g)_e:\mathfrak{g}\longrightarrow T_gG$ is the differential of $L_g$ at $e\in G$ \cite[Section 5, Equation (14L)]{Marsden}. 

Now consider the identifications
$$T_{(e,x)}(G\times\mathfrak{g})=\mathfrak{g}\oplus\mathfrak{g}\quad\text{and}\quad T_{(e,x)}^*(G\times\mathfrak{g})=(\mathfrak{g}\oplus\mathfrak{g})^*=\mathfrak{g}^*\oplus\mathfrak{g}^*$$ for each $x\in\mathfrak{g}$. Write $P_{\omega}$ for the Poisson bivector on $T^*G$ determined by $\omega$, noting that
$(P_{\omega})_{(e,x)}$ is a vector space isomorphism
$$(P_{\omega})_{(e,x)}:\mathfrak{g}^*\oplus\mathfrak{g}^*\overset{\cong}\longrightarrow\mathfrak{g}\oplus\mathfrak{g}$$ for each $x\in\mathfrak{g}$. To compute $(P_{\omega})_{(e,x)}$, let 
$$\kappa:\mathfrak{g}^*\overset{\cong}\longrightarrow\mathfrak{g}$$
denote the inverse of \eqref{Equation: Identification}. This leads to the following lemma, which will be needed later.

\begin{lem}\label{Lemma:PoissonT*G}
If $x\in\mathfrak{g}$, then
$$(P_\omega)_{(e,x)}(\alpha,\beta)  = (\kappa(\beta), [x,\kappa(\beta)]- \kappa(\alpha))$$
for all $(\alpha,\beta)\in \g^*\oplus\g^*$. 
\end{lem}
\begin{proof}
Write $P_\omega (\alpha,\beta) = (y,z)\in \g\oplus\g$ and note that
\begin{align*}\alpha(v) + \beta(w) & = \omega_{(e,x)}((P_\omega)_{(e,x)}(\alpha,\beta),(v,w))\\ & = \omega_{(e,x)}((y,z), (v,w)) \\ & = \langle y,w\rangle - \langle z,v\rangle + \langle x, [y,v]\rangle \\ & = \langle y,w\rangle + \langle [x,y]-z,v\rangle.\end{align*}
for all $v,w\in\mathfrak{g}$. It follows that 
$\kappa(\alpha)=[x,y]-z$ and $\kappa(\beta)=y$, or equivalently
$$y=\kappa(\beta)\quad\text{and}\quad z = [x,\kappa(\beta)]-\kappa(\alpha).$$
\end{proof}

Now assume that $G$ is the adjoint group of $\g$. The variety $T^*\overline{G}(\log D)$ admits a distinguished log symplectic structure (e.g. \cite{BalibanuUniversal}), some aspects of which we now describe. We begin by noting that
\begin{equation}\label{Equation: Embed}\tilde{\varphi}:T^*G\longrightarrow T^*\overline{G}(\log D),\quad (g,x)\longrightarrow ((g,e)\cdot\mathfrak{g}_{\Delta},(\mathrm{Ad}_{g}(x),x)).\end{equation}
is a symplectomorphism onto the unique open dense symplectic leaf in $T^*\overline{G}(\log D)$. This yields the commutative diagram
$$\begin{tikzcd}
T^*G \arrow[r, "\tilde{\varphi}"] \arrow[d]
& T^*\overline{G}(\log D) \arrow[d] \\
G \arrow[r, "\varphi"]
& \overline{G}
\end{tikzcd},$$
where $\varphi:G\longrightarrow\overline{G}$ is the map \eqref{Equation: Loc}. One also observes $\tilde{\varphi}$ to be equivariant with respect to \eqref{Equation: Action on log cotangent} and the following ($G\times G$)-action on $T^*G$:
\begin{equation}\label{Equation: Action on cotangent}(g_1,g_2)\cdot (h,y):=(g_1hg_2^{-1},\mathrm{Ad}_g(y)).
\end{equation}
The ($G\times G$)-actions \eqref{Equation: Action on log cotangent} and \eqref{Equation: Action on cotangent} are Hamiltonian with respective moment maps      
\begin{equation}\label{Equation: Second moment}\overline{\rho}=(\overline{\rho}_L,\overline{\rho}_R):T^*\overline{G}(\log D)\longrightarrow\g\oplus\g,\quad (\gamma,(y_1,y_2))\longrightarrow (y_1,y_2)
\end{equation}
and
\begin{equation}\label{Equation: First moment}\rho=(\rho_L,\rho_R):T^*G\longrightarrow\g\oplus\g,\quad (g,y)\longrightarrow (\mathrm{Ad}_g(y),y).\end{equation}

Now suppose that $(X,P,\nu)$ is a Hamiltonian $G$-variety. Endow $X$ with the Hamiltonian ($G\times G$)-variety structure for which $$G_R:=\{e\}\times G$$ acts trivially and $$G_L:=G\times\{e\}$$ acts via the original Hamiltonian $G$-action and the identification $G=G_L$. It follows that the product Poisson varieties $X\times T^*G$ and $X\times T^*\overline{G}(\log D)$ are Hamiltonian ($G\times G$)-varieties with respective moment maps
\begin{equation}\label{Equation: muXT*G}
\mu=(\mu_L,\mu_R):X\times T^*G\longrightarrow\g\oplus\g,\quad (x,(g,y))\longrightarrow (\nu(x)-\mathrm{Ad}_g(y),-y)
\end{equation}
and
\begin{equation}\label{Equation: mubarXT*GlogD}\overline{\mu}=(\overline{\mu}_L,\overline{\mu}_R):X\times T^*\overline{G}(\log D)\longrightarrow\g\oplus\g,\quad (x,(\gamma,(y_1,y_2)))\longrightarrow (\nu(x)-y_1,-y_2).
\end{equation}
We also have a commutative diagram
\begin{equation}\begin{tikzcd}\label{Equation: Diag}
X\times T^*G \arrow{rr}{i} \arrow[swap]{dr}{\mu}& & X\times T^*\overline{G}(\log D) \arrow{dl}{\overline{\mu}}\\
& \mathfrak{g}\oplus\mathfrak{g} & 
\end{tikzcd},\end{equation}
where 
\begin{equation}\label{Equation: Equivariant embedding} i:X\times T^*G\longrightarrow X\times T^*\overline{G}(\log D),\quad (x,(g,y))\longrightarrow (x,((g,e)\cdot\g_{\Delta},(\mathrm{Ad}_g(y),y))).\end{equation}
is the ($G\times G$)-equivariant open Poisson embedding given by the product of \eqref{Equation: Embed} with the identity map $X\longrightarrow X$.

The Hamiltonian ($G\times G$)-variety $X\times T^*G$ warrants some further discussion. One knows that the geometric quotient
$$\mu_L^{-1}(0)\longrightarrow(X\times T^*G)\sslash G_L$$ exists, and that the action of $G_R$ on $\mu_L^{-1}(0)$ descends to a Hamiltonian action of $G$ on $(X\times T^*G)\sslash G_L$. An associated moment map is obtained by descending 
$$-\mu_R\big\vert_{\mu_L^{-1}(0)}:\mu_L^{-1}(0)\longrightarrow\g$$ 
to the quotient variety $(X\times T^*G)\sslash G_L$. It is then not difficult to verify that
\begin{equation}\label{Equation: Reduction isomorphism}\psi:X\overset{\cong}\longrightarrow(X\times T^*G)\sslash G_L,\quad x\longrightarrow [x:(e,\nu(x))],\quad x\in X\end{equation}
is an isomorphism of Hamiltonian $G$-varieties.

\section{Poisson slices}\label{Section: Poisson slices}
This section develops the general theory of Poisson slices. Some emphasis is placed on properties of the Poisson slice $G\times\mathcal{S}_{\tau}$ and a larger log symplectic variety $\overline{G\times\mathcal{S}_{\tau}}$.  

\subsection{Poisson transversals and Poisson slices}\label{Subsection: Poisson slices}
Let $(X,P)$ be a Poisson variety. Given $x\in X$ and a subspace $V\subseteq T_xX$, we write $V^{\dagger}$ for the annihilator of $V$ in $T_x^*X$. Our notation suppresses the dependence of $V^{\dagger}$ on $T_xX$, as the ambient tangent space will always be clear from context. We will use an analogous notation for vector subbundles of $TX$.

Recall that a smooth locally closed subvariety $Y\subseteq X$ is called a \textit{Poisson transversal} (or \textit{cosymplectic subvariety}) if 
\begin{equation}\label{Equation: PoissonTransversal}
TX|_Y = TY \oplus P(TY^\dagger).
\end{equation}
This has the following straightforward implication for every symplectic leaf $L\subseteq X$: $L$ and $Y$ have a transverse intersection in $X$, and $L\cap Y$ is a symplectic submanifold of $L$. 

The Poisson transversal $Y$ inherits a Poisson bivector $P_Y$ from $(X,P)$. To define it, note that the decomposition \eqref{Equation: PoissonTransversal} gives rise to an inclusion $T^*Y\subseteq T^*X$. One can verify that
$$P(T^*Y)\subseteq TY,$$ and $P_Y$ is then defined to be the restriction
$$P_Y:=P\big\vert_{T^*Y}:T^*Y\longrightarrow TY.$$ Note that $Y$ need not be a Poisson subvariety of $X$ in the usual sense; restricting functions need not define a morphism $\mathcal{O}_X\longrightarrow j_*\mathcal{O}_Y$ of sheaves of Poisson algebras, where $j:Y\hookrightarrow X$ is the inclusion. This is particularly apparent if $X$ is symplectic; the Poisson transversals are the symplectic subvarieties, while the Poisson subvarieties are the open subvarieties. 

We record the following well-known fact for future reference (cf. \cite[Example 4]{Frejlich}).

\begin{prop}\label{Proposition: Symplectic transversal}
Let $X$ be a symplectic variety. If $Y\subseteq X$ is a Poisson transversal, then $Y$ is a symplectic subvariety of $X$. The resulting symplectic structure on $Y$ coincides with the Poisson structure $Y$ inherits as a transversal. 
\end{prop}

We need the following refinement in the case of log symplectic varieties.

\begin{prop}\label{Proposition: Poissonslicelogsymplectic}
	Suppose that $(X,P)$ is a log symplectic variety with divisor $Z$. Let $Y\subseteq X$ be an irreducible Poisson transversal, and write $P_{\emph{tr}}$ for the resulting Poisson bivector on $Y$.  The following statements hold.
	\begin{itemize}
		\item[(i)] The Poisson variety $(Y,P_{\emph{tr}})$ is log symplectic with divisor $Z\cap Y$. 
		\item[(ii)] If one equips $Y\setminus Z$ and $X\setminus Z$ with the symplectic structures inherited as symplectic leaves of $(Y,P_{\mathrm{tr}})$ and $(X,P)$, respectively, then $Y\setminus Z$ is a symplectic subvariety of $X\setminus Z$. 
	\end{itemize}     
\end{prop}
\begin{proof}
	We begin by proving that $Y$ is a log symplectic subvariety of $X$ in the sense of \cite[Definition 7.16]{Gualtieri}. To this end, consider the unique open dense symplectic leaf $X_{0}:=X\setminus Z\subseteq X$.
	Since $Y$ is a Poisson transversal in $X$, Proposition \ref{Proposition: Symplectic transversal} forces $Y_0:=Y\cap X_{0}$ to be a symplectic subvariety of $X_{0}$.	
	
	Now let $Z_1,\ldots,Z_k$ be the irreducible components of $Z$, and set
	$$Z_I:=\bigcap_{i\in I}Z_i$$ for each subset $I\subseteq\{1,\ldots,k\}$. Each irreducible component of $Z$ is a union of symplectic leaves in $X$ (cf. \cite[Exercise 5.2]{Pym}), implying that $Z_I$ is a union of symplectic leaves for each $I\subseteq\{1,\ldots,k\}$. On the other hand, the Poisson transversal $Y$ is necessarily transverse to the symplectic leaves in $X$. These last two sentences imply that $Y$ is transverse to $Z_I$ for all $I\subseteq\{1,\ldots,k\}$.
	
	The previous two paragraphs show $Y$ to be a log symplectic subvariety of $X$, and we let $P_{\text{log}}$ denote the resulting Poisson bivector on $Y$. It follows that $Y_{0}$ is the unique open dense symplectic leaf of $(Y,P_{\text{log}})$, and that its symplectic form is the pullback of the symplectic form on $X_{0}$. We also know that $P_{\text{tr}}$ is non-degenerate on $Y_{0}$, and that it coincides with the pullback of the symplectic structure from $X_{0}$ to $Y_{0}$ (see Proposition \ref{Proposition: Symplectic transversal}). One concludes that $P_{\text{log}}$ and $P_{\text{tr}}$ coincide on $Y_{0}$. Since $Y_0$ is dense in $Y$, it follows that $P_{\text{log}}=P_{\text{tr}}$. This establishes (i) and (ii).   
\end{proof}

The following well-known result concerns the behaviour of Poisson transversals with respect to Poisson morphisms (cf. \cite[Lemma 7]{Frejlich}).

\begin{prop}\label{Proposition: Poisson transversal}
Let $\phi: X_1\longrightarrow X_2$ be a Poisson morphism between Poisson varieties $X_1$ and $X_2$. If $Y\subseteq X_2$ is a Poisson transversal, then $\phi^{-1}(Y)$ is a Poisson transversal in $X_1$. The codimension of $\phi^{-1}(Y)$ in $X_1$ is equal to the codimension of $Y$ in $X_2$. 
\end{prop}

We now consider a concrete application of Proposition \ref{Proposition: Poisson transversal}. To this end, recall the Lie-theoretic notation and setup established in \ref{Subsection: Lie theoretic conventions}.

\begin{cor}\label{Corollary: Slice Poisson transversal}
Suppose that $(X,P,\nu)$ is a Hamiltonian $G$-variety. If $\tau = (\xi,h,\eta)$ is an $\sln_2$-triple in $\g$, then $\nu^{-1}(\mathcal S_\tau)$ is a Poisson transversal in $X$. This transversal has codimension $\dim\g-\dim(\g_\eta)$ in $X$.
\end{cor}
\begin{proof}
The moment map $\nu:X\longrightarrow \g$ is necessarily a morphism of Poisson varieties (e.g. \cite[Proposition 7.1]{Weinstein}). At the same time, \cite[Section 3.1]{Gan} explains that $\mathcal{S}_{\tau}$ is a Poisson transversal in $\g$. The desired now result now follows immediately from Proposition \ref{Proposition: Poisson transversal}.
\end{proof}

A consequence of Corollary \ref{Corollary: Slice Poisson transversal} is that $\nu^{-1}(\mathcal{S}_{\tau})$ inherits a Poisson bivector $P_{\tau}$ from $(X,P)$. This gives rise to our notion of a \textit{Poisson slice}.

\begin{defn}\label{Definition: PoissonSlice}
Suppose that $(X,P,\nu)$ is a Hamiltonian $G$-variety, and let $\tau$ be an $\mathfrak{sl}_2$-triple in $\mathfrak{g}$. We call $X_\tau:=(\nu^{-1}(\mathcal S_\tau),P_{\tau})$ the \emph{Poisson slice} of $(X,P,\nu)$ with respect to $\tau$.
\end{defn}

This next proposition explains why we call $X_\tau$ a Poisson slice; it is a slice for the $G$-action on $X$ in the following sense.

\begin{prop}\label{Proposition: PoissonSliceTransverseG}
Let $(X,P,\nu)$ be a Hamiltonian $G$-variety. If $\tau$ is an $\mathfrak{sl}_2$-triple in $\g$, then $X_\tau$ is transverse to the $G$-orbits in $X$.
\end{prop}
\begin{proof} 
Fix $x\in\nu^{-1}(\mathcal{S}_{\tau})$ and set $y:=\nu(x)\in\mathcal{S}_{\tau}$. Consider the differential
$d\nu_x:T_xX\longrightarrow\g$ and its dual $d\nu_x^*:\g^*\longrightarrow T_x^*X$, and let $P_{\mathfrak{g}}$ be the Poisson bivector on $\g$.
Since $\nu$ is a morphism of Poisson varieties, we have 
$$(P_\g)_y = d\nu_x\circ P_x\circ d\nu_x^*.$$
We also know $\mathcal S_\tau\subseteq\g$ to be a Poisson transversal (e.g. by Corollary \ref{Corollary: Slice Poisson transversal}), so that
$$\g = T_y\mathcal{S}_{\tau} \oplus (P_\g)_y((T_y\mathcal{S}_{\tau})^{\dagger}) = T_y\mathcal{S}_{\tau}\oplus d\nu_x(P_x(d\nu_x^*((T_y\mathcal{S}_{\tau})^{\dagger}))).$$
One immediate conclusion is that $\nu$ is transverse to $\mathcal S_\tau$. We also conclude that 
$$T_x(\nu^{-1}(\mathcal S_\tau)) = \ker\bigg(\mathrm{pr}_2\circ d\nu_x: T_xX\longrightarrow (P_\g)_y((T_y\mathcal{S}_{\tau})^{\dagger})\bigg),$$
where $$\mathrm{pr}_2: \g = T_y\mathcal{S}_{\tau} \oplus (P_\g)_y((T_y\mathcal{S}_{\tau})^{\dagger})\longrightarrow (P_\g)_y((T_y\mathcal{S}_{\tau})^{\dagger})$$ is the natural projection. 
It follows that  
$$T_x(\nu^{-1}(\mathcal{S}_{\tau}))^\dagger= \mathrm{image}\bigg(d\nu_x^*\circ\mathrm{pr}_2^*:  (P_\g)_y((T_y\mathcal{S}_{\tau})^{\dagger})^*\longrightarrow T_x^*X\bigg),$$
where $$\mathrm{pr}_2^*:(P_\g)_y((T_y\mathcal{S}_{\tau})^{\dagger})^*\longrightarrow\g^*$$ is the dual of $\mathrm{pr}_2$. This amounts to the statement that
$$T_x(\nu^{-1}(\mathcal{S}_{\tau}))^\dagger=d\nu_x^*(\g_\eta^{\dagger}),$$ while we know that 
the Killing form identifies $\g_{\eta}^{\dagger}\subseteq\g^*$ with $\g_{\eta}^{\perp}=[\g,\eta]\subseteq\g$.
We conclude that
$$T_x(\nu^{-1}(\mathcal{S}_{\tau}))^\dagger = \mathrm{span}\{(d\nu^{[\eta,b]})_x\colon b\in\g\},$$ where $\nu^{[\eta,b]}:X\longrightarrow\mathbb{C}$ is defined by
$$\nu^{[\eta,b]}(z)=\langle\nu(z),[\eta,b]\rangle.$$
Equations \eqref{Equation: HamiltonianVectorField} and \eqref{Equation: MomentCondition} now imply that
$$P_x(T_x(\nu^{-1}(\mathcal{S}_{\tau}))^\dagger) = \mathrm{span}\{P_x((d\nu^{[\eta,b]})_x)\colon b\in\g\} = \mathrm{span}\{V^{[\eta,b]}_x\colon b\in\g\}\subseteq T_x(Gx).$$
This combines with $\nu^{-1}(\mathcal{S}_{\tau})$ being a Poisson transversal to yield 
$$T_xX = T_x(\nu^{-1}(\mathcal S_\tau))\oplus P_x(T_x(\nu^{-1}(\mathcal{S}_{\tau}))^\dagger) = T_x(\nu^{-1}(\mathcal S_\tau))+ T_x(Gx),$$
completing the proof.
\end{proof}

Let $Y$ be an irreducible component of $X_{\tau}$. The bivector $P_{\tau}$ then restricts to a Poisson bivector $P_{Y,\tau}$ on $Y$. This leads to the following observation.

\begin{cor}\label{Corollary: Log symplectic}
	Suppose that $(X,P,\nu)$ is a Hamiltonian $G$-variety. Assume that $(X,P)$ is log symplectic with divisor $Z$, and let $\tau$ be an $\mathfrak{sl}_2$-triple in $\g$. Let $Y$ be an irreducible component of the Poisson slice $X_{\tau}$. 
	\begin{itemize}
		\item[(i)] The Poisson variety $(Y,P_{Y,\tau})$ is log symplectic with divisor $Y\cap Z$.
		\item[(ii)] If one equips $Y\setminus Z$ and $X\setminus Z$ with the symplectic structures inherited as symplectic leaves of $(Y,P_{Y,\tau})$ and $(X,P)$, respectively, then $Y\setminus Z$ is a symplectic subvariety of $X\setminus Z$.
		\item[(iii)] If $(X,P)$ is symplectic, then $(X_{\tau},P_{\tau})$ is symplectic and the symplectic form on $(X,P)$ pulls back to the symplectic form on $(X_{\tau},P_{\tau})$. 
	\end{itemize}
\end{cor}

\begin{proof}
	This follows immediately from Proposition \ref{Proposition: Symplectic transversal}, Proposition \ref{Proposition: Poissonslicelogsymplectic}, and Corollary \ref{Corollary: Slice Poisson transversal}.
\end{proof}

The following immediate consequence is used extensively in later sections.

\begin{cor}\label{Corollary: Symplectic subvariety}
If $\tau$ is an $\sln_2$-triple in $\g$, then $G\times\mathcal S_\tau$ is a symplectic subvariety of $T^*G=G\times\g$.
\end{cor}
\begin{proof}
	Apply Corollary \ref{Corollary: Log symplectic}(iii) to $X = T^*G$ with the Hamiltonian action of $G_R=\{e\}\times G\subseteq G\times G$.
\end{proof}

\subsection{Poisson slices via Hamiltonian reduction}\label{Subsection: Poisson slices via Hamiltonian reduction}
Recall the Hamiltonian action of $G\times G$ on $T^*G= G\times\mathfrak{g}$ discussed in Subsection \ref{Subsection: MomentLogCotangent}. The symplectic subvariety $G\times\mathcal{S}_{\tau}$ is invariant under $G_L=G\times\{e\}\subseteq G\times G$, and 
\begin{equation}\label{Equation: Symplectic moment map}\rho_{\tau}:=\rho_L\big\vert_{G\times\mathcal{S}_{\tau}}:G\times\mathcal{S}_{\tau}\longrightarrow\mathfrak{g},\quad (g,x)\longrightarrow\mathrm{Ad}_g(x)\end{equation} is a corresponding moment map. Now let $(X,P,\nu)$ be a Hamiltonian $G$-variety, and consider the product Poisson variety $X\times (G\times \mathcal S_\tau)$. The diagonal action of $G$ on $X\times (G\times\mathcal{S}_{\tau})$ is then Hamiltonian with moment map
$$\mu_{\tau}:X\times (G\times\mathcal{S}_{\tau})\longrightarrow\mathfrak{g},\quad (x,(g,y))\longrightarrow\nu(x)-\mathrm{Ad}_g(y).$$ These considerations allow us to realize Poisson slices via Hamiltonian reduction. 

\begin{prop}\label{Proposition: PoissonSlice}
Let $(X,P,\nu)$ be a Hamiltonian $G$-variety, and let $\tau$ be an $\sln_2$-triple in $\g$. If we endow $X\times (G\times\mathcal{S}_{\tau})$ with the Poisson structure and Hamiltonian $G$-action described above, then there is a Poisson variety isomorphism 
	\begin{equation}\label{Equation: Ham iso}\psi_{\tau}:X_{\tau}\overset{\cong}\longrightarrow (X\times (G\times \mathcal S_\tau))\sslash G,\quad x\longrightarrow [x:(e,\nu(x))].\end{equation} 
\end{prop}
\begin{proof}
We begin by noting that 
\begin{align*}
\mu_{\tau}^{-1}(0) & = \{(x,(g,y))\in X\times (G\times\mathcal S_\tau)\colon \nu(x)=\Ad_g(y)\} \\ & = \{(x,(g,y))\in X\times (G\times\mathcal S_\tau)\colon \nu(g^{-1}\cdot x) = y\}.
\end{align*}
It follows that the $G$-invariant map
$$J:X\times (G\times \mathcal S_\tau) \longrightarrow X,\quad (x,(g,y))\longrightarrow g^{-1}\cdot x$$
satisfies $J(\mu_{\tau}^{-1}(0)) \subseteq \nu^{-1}(\mathcal S_\tau)=X_{\tau}$, thereby inducing a map  
$$\pi := J\big\vert_{\mu_{\tau}^{-1}(0)}: \mu_{\tau}^{-1}(0)\longrightarrow X_{\tau}.$$
One then verifies that
$$\pi^{-1}(x)=G\cdot(x,e,\nu(x))\subseteq X\times (G\times\mathcal{S}_{\tau})$$ 
for all $x\in X_{\tau}$, where $G\cdot(x,e,\nu(x))$ is the $G$-orbit of $(x,e,\nu(x))$ in $X\times (G\times\mathcal{S}_{\tau})$. This forces $\pi$ to be the geometric quotient of $\mu_{\tau}^{-1}(0)$ by $G$ (e.g. by \cite[Proposition 25.3.5]{Tauvel}), i.e. 
$$(X\times (G\times \mathcal S_\tau))\sslash G = X_{\tau}.$$

We now have two Poisson structures on $X_{\tau}$: the Poisson structure $P_{\text{red}}$ from Hamiltonian reduction, and the structure $P_{\text{tr}}$ obtained from $X_{\tau}$ being a Poisson slice in $X$. It suffices to show that these Poisson structures coincide.

Fix $x\in X_{\tau}$ and $\alpha\in T_x^*X_{\tau}$. Since $X_{\tau}$ is a Poisson transversal in $X$, there is a unique extension of $\alpha$ to an element 
$$ \tilde{\alpha}\in \bigg(P_x\big((T_xX_{\tau})^\dagger\big)\bigg)^\dagger\subseteq T_x^*X.$$ 
The discussion of Poisson transversals in Subsection \ref{Subsection: Poisson slices} then implies that 
\begin{equation}\label{Equation: P3}(P_{\text{tr}})_x(\alpha) = P_x(\tilde\alpha).\end{equation}
We also have
\begin{equation}\label{Equation: P1}(P_{\text{red}})_x(\alpha) = d\pi_z ((P_{\tau})_z(\tilde\alpha')),\end{equation}
where $z=(x,e,\nu(x))$, 
$$\tilde\alpha'\in T_z(Gz)^\dagger\subseteq T_{z}^*(X\times (G\times\mathcal S_\tau))$$ 
is an extension of $d\pi_z^*(\alpha)$, and 
$$d\pi_z^*:T^*_xX_{\tau}\longrightarrow T_z^*(\mu_{\tau}^{-1}(0))$$ 
is the dual of 
$$d\pi_z:T_z(\mu_{\tau}^{-1}(0))\longrightarrow T_xX_{\tau}.$$
Since $J$ is $G$-invariant, we may take 
$$\tilde\alpha' := dJ_z^*(\tilde\alpha).$$
We also observe that
$$dJ_{z}(a,b,c) = a - (V^{b})_x$$
for all $(a,b,c)\in T_z(X\times (G\times\mathcal{S}_{\tau}))=T_xX\oplus\mathfrak{g}\oplus\mathfrak{g}_{\eta}$, where $V^b$ is the fundamental vector field on $X$ associated to $b\in\g$.
It follows that 
$$(dJ_z^*(\tilde\alpha))(a,b,c) = \tilde\alpha(a)-\tilde\alpha((V^b)_x) = \tilde\alpha(a)-\tilde\alpha(P_x((d\nu^b)_x)) =  \tilde\alpha(a)+(d\nu^b)_x(P_x(\tilde\alpha)),$$ yielding 
\begin{equation}\label{Equation: P2}\tilde\alpha' = (\tilde\alpha, d\nu_x(P_x(\tilde\alpha)) , 0)\in T^*_{z}(X\times (G\times \mathcal S_\tau)) = T_x^*X \oplus\g^*\oplus \g_\eta^*=T_x^*X\oplus\mathfrak{g}\oplus\mathfrak{g}_{\xi},
\end{equation}   
where we have made the identifications $\g_\eta^* = (\g/[\g,\xi])^*= [\g,\xi]^\perp = \g_\xi$.
Now set $w=(e,\nu(x))\in G\times\mathcal{S}_{\tau}$ and let $Q_{\tau}$ be the Poisson bivector on $G\times\mathcal{S}_{\tau}$. Lemma \ref{Lemma:PoissonT*G} then gives
$$(Q_{\tau})_w(d\nu_x(P_x(\tilde\alpha)),0) = (0,-d\nu_x(P_x(\tilde\alpha))).$$
This combines with \eqref{Equation: P3}, \eqref{Equation: P1}, and \eqref{Equation: P2} to yield 
\begin{align*}(P_{\text{red}})_x(\alpha) & = d\pi_z(P_x(\tilde\alpha),-(Q_{\tau})_w(d\nu_x(P_x(\tilde\alpha)),0))\\ & = d\pi_z(P_x(\tilde\alpha),0,d\nu_x(P_x(\tilde\alpha)))\\ & = P_x(\tilde\alpha)\\ &=(P_{\text{tr}})_x(\alpha),\end{align*}
as desired.
\end{proof}  

\begin{rem}\label{Rem: tauNull}
In the special case $\tau = 0$, we have $\mathcal{S}_{\tau}=\g$ and $G\times\mathcal S_\tau = G\times \g = T^*G$. Proposition \ref{Proposition: PoissonSlice} is then seen to recover the isomorphism \eqref{Equation: Reduction isomorphism}.
\end{rem}

Our next result is that Poisson slices can be realized via Hamiltonian reduction with respect to unipotent radicals of parabolic subgroups. To formulate this result, let $\tau = (\xi,h,\eta)$ be an $\sln_2$-triple in $\g$ and write $\g_\lambda\subseteq \g$ for the eigenspace of $\ad_h$ with eigenvalue $\lambda\in\mathbb{Z}$. The parabolic subalgebra 
$$\p_\tau := \bigoplus_{\lambda\leq 0}\g_\lambda$$
then has
$$\uu_\tau := \bigoplus_{\lambda <0}\g_\lambda$$
as its nilradical. Now consider the identifications 
$$\uu_\tau^* \cong \g/\uu_\tau^\perp=\mathfrak{g}/\mathfrak{p}_{\tau} \cong \uu_\tau^- := \bigoplus_{\lambda >0}\g_\lambda,$$
and thereby regard $\xi\in\mathfrak{u}_{\tau}^{-}$ as an element of $\mathfrak{u}_{\tau}^*$. Write $U_{\tau}\subseteq G$ for the unipotent subgroup with Lie algebra $\mathfrak{u}_{\tau}$, and let $(U_{\tau})_{\xi}$ be the $U_{\tau}$-stabilizer of $\xi$ under the coadjoint action. 

\begin{rem}
The Lie algebra of $(U_{\tau})_{\xi}$ is given by
$$(\mathfrak{u}_{\tau})_{\xi}=\bigoplus_{\lambda\leq -2}\mathfrak{g}_{\lambda}.$$ It follows that $(U_{\tau})_{\xi}=U_{\tau}$ if and only if $\tau$ is an even $\mathfrak{sl}_2$-triple, i.e. $\mathfrak{g}_{-1}=\{0\}$. If $\tau$ is a principal triple, then $\tau$ is even and $(U_{\tau})_{\xi}=U_{\tau}$ is a maximal unipotent subgroup of $G$.
\end{rem}

Let $(X,P,\nu)$ be a Hamiltonian $G$-variety. The action of $U_\tau$ is also Hamiltonian with moment map $\nu_\tau := p_\tau\circ \mu$, where 
$$\g = \p_\tau\oplus\uu_\tau^-\overset{p_{\tau}}\longrightarrow\uu_\tau^-=\mathfrak{u}_{\tau}^*$$
 is the projection. 
One has 
$$\nu_{\tau}^{-1}(\xi)=\nu^{-1}(\xi+\mathfrak{p}_{\tau}),$$
while the proof of \cite[Lemma 3.2]{Bielawski} shows the stabilizer $(U_{\tau})_{\xi}$ to act freely on $\xi+\mathfrak{p}_{\tau}$. It follows that $(U_{\tau})_{\xi}$ acts freely on $\nu_{\tau}^{-1}(\xi)$. This leads us to prove Proposition \ref{Proposition:PoissonSlice}, i.e. that the geometric quotient
\begin{equation}\label{Equation: Geometric quotient}X\sslash_\xi U_\tau = \nu_\tau^{-1}(\xi)/(U_\tau)_\xi\end{equation} exists and is Poisson-isomorphic to $X_\tau$.

\begin{rem}
The type of Hamiltonian reduction performed in \eqref{Equation: Geometric quotient} is particularly well-studied in the case of a principal triple $\tau$. In this case, one sometimes calls the Poisson variety $X\sslash_\xi U_\tau$ a \textit{Whittaker reduction} (e.g. \cite{Bezrukavnikov,Dimofte}). The nomenclature reflects Kostant's result \cite[Theorem 1.2]{KostantWhittaker}. 
\end{rem}
 
\begin{prop}\label{Proposition:PoissonSlice}
Let $(X,P,\nu)$ be a Hamiltonian $G$-variety. If $\tau=(\xi,h,\eta)$ is an $\mathfrak{sl}_2$-triple in $\g$, then there is a canonical isomorphism
$$X\sslash_\xi U_\tau \cong X_\tau$$
of Poisson varieties.
\end{prop}
\begin{proof}
We begin by exhibiting $X_{\tau}$ as the geometric quotient of $\nu_{\tau}^{-1}(\xi)$ by $(U_{\tau})_{\xi}$. To this end, the proof of \cite[Lemma 3.2]{Bielawski} explains that
$$(U_{\tau})_{\xi}\times\mathcal{S}_{\tau}\longrightarrow\xi+\mathfrak{p}_{\tau},\quad (u,x)\longrightarrow\mathrm{Ad}_u(x)$$ 
defines a variety isomorphism. Composing the inverse of this isomorphism with the projection
$$(U_{\tau})_{\xi}\times\mathcal{S}_{\tau}\longrightarrow (U_{\tau})_{\xi}$$ 
then yields a map
$$\phi:\xi+\mathfrak{p}_{\tau}\longrightarrow (U_{\tau})_{\xi}.$$ 
Note that for $y\in\xi+\mathfrak{p}_{\tau}$, $\phi(y)$ is the unique element of $(U_{\tau})_{\xi}$ satisfying
$$\mathrm{Ad}_{\phi(y)^{-1}}(y)\in\mathcal{S}_{\tau}.$$
We may therefore define the map
$$\nu_{\tau}^{-1}(\xi)=\nu^{-1}(\xi+\mathfrak{p}_{\tau})\overset{\theta}\longrightarrow X_{\tau},\quad x\longrightarrow (\phi(\nu(x)))^{-1}\cdot x.$$
One has
$$\theta^{-1}(x)=(U_{\tau})_{\xi}\cdot x$$ for all $x\in\nu_{\tau}^{-1}(\xi)$, and we deduce that $\theta$ is the geometric quotient of $\nu_{\tau}^{-1}(\xi)$ by $(U_{\tau})_{\xi}$ (e.g. by \cite[Proposition 25.3.5]{Tauvel}). 

The previous paragraph establishes the following fact: Hamiltonian reductions of Hamiltonian $G$-varieties by $U_{\tau}$ at level $\xi$ always exist as geometric quotients. We implicitly use this observation in several places below.

To see that the Poisson structures on $X_\tau$ and $X\sslash_\xi U_\tau$ coincide, we argue as follows. One has a canonical isomorphism 
\begin{equation}\label{Equation: GStauKostantWhittaker} T^*G\sslash_\xi U_\tau\cong G\times\mathcal S_\tau\end{equation}
of symplectic varieties, where $U_{\tau}$ acts on $T^*G$ via \eqref{Equation: Action on cotangent} as the subgroup $U_{\tau}=\{e\}\times U_{\tau}\subseteq G\times G$ (see \cite[Lemma 3.2]{Bielawski}). Note also that $T^*G\sslash_\xi U_\tau$ and $G\times\mathcal S_\tau$ come with Hamiltonian actions of $G$ induced by the action of $G_L=G\times\{e\}$ on $T^*G\cong G\times\g$. One then readily verifies that \eqref{Equation: GStauKostantWhittaker} is an isomorphism of Hamiltonian $G$-varieties. 

Proposition \ref{Proposition: PoissonSlice} gives a canonical isomorphism of Poisson varieties 
$$X_\tau \cong (X\times (G\times \mathcal S_\tau))\sslash G.$$
The previous paragraph allows us to write this isomorphism as
$$X_\tau \cong (X\times (T^*G\sslash_\xi U_\tau))\sslash G=((X\times T^*G)\sslash_{\xi}U_{\tau})\sslash G,$$ where $U_{\tau}$ acts trivially on $X$. 
Since the actions of $G$ and $U_{\tau}$ on $X\times T^*G$ commute with one another, it follows that 
$$X_\tau \cong ((X\times T^*G)\sslash G)\sslash_\xi U_\tau.$$ 
An application of Remark \ref{Rem: tauNull} then yields
$$X_\tau \cong X\sslash_\xi U_\tau,$$
completing the proof.
\end{proof}

\subsection{Poisson slices in the log cotangent bundle of $\overline{G}$}\label{Subsection: KW}
Fix an $\sln_2$-triple $\tau$ in $\mathfrak{g}$ and recall the notation in Subsection \ref{Subsection: MomentLogCotangent}. In what follows, we study the Poisson slice
$$\overline{G\times\mathcal{S}_{\tau}}:=\overline{\rho}_R^{-1}(\mathcal{S}_\tau)\subseteq T^*\overline{G}(\log D)$$  
and its properties. We begin by observing that \begin{equation}\label{Equation: first} \overline{G\times\mathcal{S}_{\tau}} = \{(\gamma,(x,y))\in\overline{G}\times(\g\oplus\g)\colon (x,y)\in \gamma\text{ and }y\in\mathcal{S}_\tau\}.\end{equation}

A few simplifications arise if $\tau$ is a principal $\mathfrak{sl}_2$-triple. To this end, recall the adjoint quotient $$\chi:\g\longrightarrow\mathrm{Spec}(\mathbb{C}[\g]^G)$$ and the associated concepts and notation  discussed in Subsection \ref{Subsection: Lie theoretic conventions}. The image of $\overline{\rho}:T^*\overline{G}(\log D)\longrightarrow\g\oplus\g$ is known to be \begin{equation}\label{Equation: Image rho}\mathrm{image}(\overline{\rho})=\{(x,y)\in\g\oplus\g\colon \chi(x) = \chi(y)\}\end{equation}
(see \cite[Proposition 3.4]{BalibanuUniversal}). One consequence is that $x,y\in\g$ lie in the same fibre of $\chi$ whenever $(x,y)\in\gamma$ for some $\gamma\in\overline{G}$. Since $\mathcal{S}_{\tau}$ is a section of $\chi$, this fact combines with \eqref{Equation: first} to yield 
\begin{equation}\label{Equation: mu2Stau}
\overline{G\times\mathcal{S}_{\tau}} = \{(\gamma,(x,x_{\tau})):\gamma\in\overline{G},\text{ }x\in\g,\text{ and } (x,x_{\tau})\in \gamma\}.
\end{equation}

We now develop some more manifestly geometric properties of $\overline{G\times\mathcal{S}_{\tau}}$, beginning with the following result.

\begin{thm}\label{Theorem: irredgeneral}
If $\tau$ is an $\mathfrak{sl}_2$-triple in $\g$, then $\overline{G\times\mathcal{S}_{\tau}}$ is irreducible.
\end{thm}

\begin{proof}
	Consider the closed subvariety
	$$Y:=\{(x,y)\in \g\times\mathcal S_\tau :\chi(x) = \chi_\tau(y)\}\subseteq\g\oplus\g,$$
	where $\chi_{\tau}:\mathcal{S}_{\tau}\longrightarrow\mathrm{Spec}(\mathbb{C}[\g]^G)$ denotes the restriction of $\chi$ to $\mathcal{S}_{\tau}$.
	It follows from \eqref{Equation: first} and \eqref{Equation: Image rho} that
	\begin{equation}\label{Equation: Useful map}\overline{G\times\mathcal{S}_{\tau}}\longrightarrow Y,\quad (\gamma,(x,y))\longrightarrow (x,y)\end{equation} is the pullback of $\overline{\rho}:T^*\overline{G}(\log D)\longrightarrow\g\oplus\g$ along the inclusion $Y\hookrightarrow\g\oplus\g$, and that \eqref{Equation: Useful map} is surjective. One also knows that $\overline{\rho}$ is proper, as it results from restricting the natural projection $\overline{G}\times(\g\oplus\g)\longrightarrow\g\oplus\g$ to $T^*\overline{G}(\log D)\subseteq \overline{G}\times(\g\oplus\g)$. The surjection \eqref{Equation: Useful map} is therefore proper, while the proof of \cite[Proposition 3.11]{BalibanuUniversal} shows \eqref{Equation: Useful map} to have connected fibres. If $Y$ were connected, then the previous sentence would force $\overline{G\times\mathcal{S}_{\tau}}$ to be connected as well. This would in turn force $\overline{G\times\mathcal{S}_{\tau}}$ to be irreducible, as Poisson slices are smooth. 
	
	In light of the previous paragraph, it suffices to prove that $Y$ is irreducible. We begin by decomposing $\g$ into its simple factors, i.e. $$\g = \g_1\oplus\cdots\oplus\g_N$$ with each $\g_i$ a simple Lie algebra. Our $\mathfrak{sl}_2$-triple $\tau$ then amounts to having an $\mathfrak{sl}_2$-triple $\tau_i$ in $\g_i$ for each $i=1,\ldots,N$, yielding $$\mathcal S_{\tau} = \mathcal S_{\tau_1}\times\dots\times\mathcal S_{\tau_N}\subseteq\g_1\oplus\cdots\oplus\g_N.$$ It also follows that $\chi_{\tau}$ decomposes as a product $$\chi_\tau = (\chi_1)_{\tau_1}\times\dots\times(\chi_N)_{\tau_N},$$ where $\chi_i$ is the adjoint quotient map on $\g_i$ and $(\chi_i)_{\tau_i}$ is its restriction to $\mathcal{S}_{\tau_i}$. The results \cite[Corollary 7.4.1]{Slodowy} and \cite[Theorem 5.4]{PremetSpecial} then imply that each $(\chi_i)_{\tau_i}$ is faithfully flat with irreducible fibres of dimension $\dim(\mathcal S_{\tau_i}) -\mathrm{rank}(\g_i)$. These last two sentences imply that $\chi_{\tau}$ is faithfully flat with irreducible, equidimensional fibres, and the same argument forces $\chi$ to be faithfully flat with irreducible, equidimensional fibres. Since fibred products of faithfully flat morphisms are faithfully flat, we conclude that 
	$$\tilde\chi\colon Y\longrightarrow\mathrm{Spec}(\mathbb{C}[\g]^G), \quad (x,y)\mapsto \chi(x)$$
	is faithfully flat. We also conclude that 
	$$\tilde{\chi}^{-1}(t) = \chi^{-1}(t)\times \chi_\tau^{-1}(t)$$ must be irreducible for all $t\in\mathrm{Spec}(\mathbb{C}[\g]^G)$, and that its dimension must be independent of $t$. In other words, $\tilde{\chi}$ is a faithfully flat morphism with irreducible, equidimensional fibres. This combines with the irreducibility of $\mathrm{Spec}(\mathbb{C}[\g]^G)$ and \cite[Corollary 9.6]{Hartshorne} to imply that $Y$ is pure-dimensional. We may now apply the result in \cite{Mustata} and deduce that $Y$ is irreducible. This completes the proof.
\end{proof}

\begin{cor}
If $\tau$ is an $\mathfrak{sl}_2$-triple in $\g$, then $\overline{G\times\mathcal{S}_{\tau}}$ is log symplectic.
\end{cor}

\begin{proof}
This is an immediate consequence of Corollary \ref{Corollary: Log symplectic}(i) and Theorem \ref{Theorem: irredgeneral}.
\end{proof}

Now observe that the Hamiltonian action of $G_L=G\times\{e\}\subseteq G\times G$ on $T^*\overline{G}(\log D)$ restricts to a Hamiltonian action of $G$ on $\overline{G\times\mathcal{S}_{\tau}}$. An associated moment map is given by
$$\overline{\rho}_{\tau}:=\overline{\rho}_L\bigg\vert_{\overline{G\times\mathcal{S}_{\tau}}}:\overline{G\times\mathcal{S}_{\tau}}\longrightarrow\g,\quad (\gamma,(x,y))\mapsto x.$$ At the same time, recall the Hamiltonian $G$-variety structure on $G\times\mathcal{S}_{\tau}$ and the moment map $\rho_{\tau}:G\times\mathcal{S}_{\tau}\longrightarrow\g$ discussed in Subsection \ref{Subsection: Poisson slices via Hamiltonian reduction}. Let us also recall the map $\tilde{\varphi}:T^*G\longrightarrow T^*\overline{G}(\log D)$ from \eqref{Equation: Embed}.

\begin{prop}\label{Proposition: nice}
Let $\tau$ be an $\mathfrak{sl}_2$-triple in $\g$.
\begin{itemize}
\item[(i)] The map $\tilde{\varphi}:T^*G\longrightarrow T^*\overline{G}(\log D)$ restricts to a $G$-equivariant symplectomorphism from $G\times\mathcal{S}_{\tau}$ to the unique open dense symplectic leaf in $\overline{G\times\mathcal{S}_{\tau}}$.
\item[(ii)] The diagram \begin{equation}\begin{tikzcd}\label{Equation: Equivariant diagram}
G\times\mathcal{S}_\tau \arrow{rr}{\tilde{\varphi}\big\vert_{G\times\mathcal{S}_\tau}} \arrow[swap]{dr}{\rho_{\tau}}& & \overline{G\times\mathcal{S}_\tau} \arrow{dl}{\overline{\rho}_{\tau}}\\
& \mathfrak{g} & 
\end{tikzcd}\end{equation} commutes.
\end{itemize}
\end{prop}

\begin{proof}
By Corollary \ref{Corollary: Log symplectic}, the open dense symplectic leaf in $\overline{G\times\mathcal{S}_{\tau}}$ is obtained by intersecting $\overline{G\times\mathcal{S}_{\tau}}$ with the open dense symplectic leaf in $T^*\overline{G}(\log D)$. The latter leaf is $\tilde{\varphi}(T^*G)$, as is explained in Subsection \ref{Subsection: MomentLogCotangent}. It is also straightforward to establish that
$$\tilde{\varphi}(G\times\mathcal{S}_\tau)=\overline{G\times\mathcal{S}_{\tau}}\cap\tilde{\varphi}(T^*G).$$ These last two sentences show $\tilde{\varphi}(G\times\mathcal{S}_{\tau})$ to be the unique open dense symplectic leaf in $\overline{G\times\mathcal{S}_{\tau}}$. We also know that $\tilde{\varphi}$ restricts to a symplectomorphism from $G\times\mathcal{S}_\tau$ to $\tilde{\varphi}(G\times\mathcal{S}_\tau)$, where the symplectic form on $\tilde{\varphi}(G\times\mathcal{S}_\tau)$ is the pullback of the symplectic form on the leaf in $T^*\overline{G}(\log D)$ (see Corollary \ref{Corollary: Symplectic subvariety}). It now follows from Corollary \ref{Corollary: Log symplectic}(ii) that
$$\tilde{\varphi}\big\vert_{G\times\mathcal{S}_\tau}:G\times\mathcal{S}_\tau\longrightarrow\tilde{\varphi}(G\times\mathcal{S}_\tau)$$
is a symplectomorphism with respect to this symplectic structure $\tilde{\varphi}(G\times\mathcal{S}_\tau)$ inherits as a leaf in $\overline{G\times\mathcal{S}_{\tau}}$. This symplectomorphism is $G$-equivariant, as $\tilde{\varphi}:T^*G\longrightarrow T^*\overline{G}(\log D)$ is ($G\times G$)-equivariant. The proof of (i) is therefore complete, while a straightforward calculation yields (ii).
\end{proof}

\begin{rem}\label{Remark: Projective fibres}
	Let $\tau$ be a principal $\mathfrak{sl}_2$-triple in $\g$. The description \eqref{Equation: mu2Stau} allows one to define a closed embedding
	$$\overline{G\times\mathcal{S}_\tau}\longrightarrow\overline{G}\times\g,\quad(\gamma,(x,x_{\tau}))\longrightarrow (\gamma,x).$$ We thereby obtain a commutative diagram
	$$\begin{tikzcd}
	\overline{G\times\mathcal{S}_\tau} \arrow{rr} \arrow[swap]{dr}{\overline{\rho}_{\tau}}& & \overline{G}\times\g \arrow{dl}\\
	& \mathfrak{g} &  
	,\end{tikzcd}$$
	where $\overline{G}\times\g\longrightarrow\g$ is projection to the second factor. One immediate consequence is that $\overline{\rho}_{\tau}$ has projective fibres, so that \eqref{Equation: Equivariant diagram} realizes $\overline{\rho}_{\tau}$ as a fibrewise compactification of $\rho_{\tau}$. It also follows that 
	$$\overline{\rho}_{\tau}^{-1}(x)\longrightarrow\{\gamma\in\overline{G}:(x,x_{\tau})\in\gamma\},\quad (\gamma,(x,x_{\tau}))\longrightarrow\gamma$$ 
	is a variety isomorphism for each $x\in\g$. 
\end{rem}

\subsection{Relation to the universal centralizer and its fibrewise compactification}\label{Subsection: Relation} 
Let $\tau$ be a principal $\mathfrak{sl}_2$-triple in $\g$. It is instructive to examine the relationship between $G\times\mathcal{S}_{\tau}$ and $\overline{G\times\mathcal{S}_{\tau}}$ in the context of Balib\u{a}nu's paper \cite{BalibanuUniversal}. We begin by recalling that the \textit{universal centralizer} of $\g$ is the closed subvariety of $T^*G=G\times\g$ defined by
$$\mathcal{Z}_{\g}^{\tau}:=\{(g,x)\in G\times\g:x\in\mathcal{S}_{\tau}\text{ and }g\in G_x\},$$ where $G_x$ is the $G$-stabilizer of $x\in\g$. At the same time, recall the Hamiltonian action of $G\times G$ on $T^*G$ and moment map $\rho:T^*G\longrightarrow\g\oplus\g$ discussed in Subsection \ref{Subsection: MomentLogCotangent}. Consider the product $\mathcal{S}_{\tau}\times\mathcal{S}_{\tau}\subseteq\g\oplus\g$ and observe that
$$\mathcal{Z}_{\mathfrak{g}}^{\tau}=\rho^{-1}(\mathcal{S}_{\tau}\times\mathcal{S}_{\tau}).$$
Note also that $\mathcal{S}_{\tau}\times\mathcal{S}_{\tau}$ is the Slodowy associated to the $\mathfrak{sl}_2$-triple $((\xi,\xi),(h,h),(\eta,\eta))$. It follows that $\mathcal{Z}_{\g}^ {\tau}$ is a Poisson slice in $T^*G$. Corollary \ref{Corollary: Log symplectic}(iii) then forces this Poisson slice to be a symplectic subvariety of $T^*G$. 

\begin{rem}\label{Remark: Previous}
	Some papers realize the symplectic structure on $\mathcal{Z}_{\mathfrak{g}}^{\tau}$ via a Whittaker reduction of $T^*G$ (e.g. \cite{BalibanuUniversal}). To this end, let $U\subseteq G$ be the unipotent subgroup with Lie algebra $\mathfrak{u}$. Proposition \ref{Proposition:PoissonSlice} then gives a canonical isomorphism
	$$\mathcal{Z}_{\mathfrak{g}}^{\tau}=\rho^{-1}(\mathcal{S}_{\tau}\times\mathcal{S}_{\tau})\cong T^*G\sslash_{(\xi,\xi)}U\times U$$
	of symplectic varieties, where the symplectic structure on $\mathcal{Z}_{\mathfrak{g}}^{\tau}$ is as defined in the previous paragraph. 
\end{rem}

One may replace $\rho:T^*G\longrightarrow\g\oplus\g$ with $\overline{\rho}:T^*\overline{G}(\log D)\longrightarrow\g\oplus \g$ and proceeed analogously. In the interest of being more precise, consider the Poisson slice
$$\overline{\mathcal{Z}_{\g}^{\tau}}:=\overline{\rho}^{-1}(\mathcal{S}_{\tau}\times\mathcal{S}_{\tau})=\{(\gamma,(x,x)):\gamma\in\overline{G},\text{ }x\in\mathcal{S}_{\tau},\text{ and }(x,x)\in\gamma\}$$ in $T^*\overline{G}(\log D)$.

\begin{rem}
	A counterpart of Remark \ref{Remark: Previous} is that Proposition \ref{Proposition:PoissonSlice} gives a canonical isomorphism
	$$\overline{\mathcal{Z}_{\mathfrak{g}}^{\tau}}=\overline{\rho}^{-1}(\mathcal{S}_{\tau}\times\mathcal{S}_{\tau})\cong T^*\overline{G}(\log D)\sslash_{(\xi,\xi)}U\times U$$ of Poisson varieties. This realization of $\overline{\mathcal{Z}_{\mathfrak{g}}^{\tau}}$ via Whittaker reduction is used to great effect in \cite{BalibanuUniversal}. 
\end{rem}

Now recall the embedding $\tilde{\varphi}:T^*G\longrightarrow T^*\overline{G}(\log D)$ discussed in Subsection \ref{Subsection: MomentLogCotangent}. Balib\u{a}nu \cite{BalibanuUniversal} shows $\overline{\mathcal{Z}_{\g}^{\tau}}$ to be log symplectic (cf. Corollary \ref{Corollary: Log symplectic}), and that $\tilde{\varphi}$ restricts to a symplectomorphism from $\mathcal{Z}_{\g}^{\tau}$ to the unique open dense symplectic leaf in $\overline{\mathcal{Z}_{\g}^{\tau}}$. One also has a commutative diagram
\begin{equation}\begin{tikzcd}\label{Equation: Non-equivariant diagram}
\mathcal{Z}_{\g}^{\tau} \arrow{rr}{\tilde{\varphi}\big\vert_{\mathcal{Z}_{\g}^{\tau}}} \arrow[swap]{dr}{q_{\tau}}& & \overline{\mathcal{Z}_{\g}^{\tau}} \arrow{dl}{\overline{q}_{\tau}}\\
& \mathcal{S}_{\tau} & 
\end{tikzcd},\end{equation}
where $$q_{\tau}(g,x)=x\quad\text{and}\quad\overline{q}_{\tau}(\gamma,(x,x))=x.$$ This diagram is seen to be the pullback of \eqref{Equation: Equivariant diagram} along the inclusion $\mathcal{S}_{\tau}\hookrightarrow\g$, and it thereby exhibits $\overline{q_{\tau}}$ as a fibrewise compactification of $q_{\tau}$ (see Remark \ref{Remark: Projective fibres} and cf. \cite[Section 3]{BalibanuUniversal}). This amounts to \eqref{Equation: Non-equivariant diagram} being the restriction of \eqref{Equation: Equivariant diagram} to a morphism between the Poisson slices 
$$\mathcal{Z}_{\g}^{\tau}=\rho^{-1}(\mathcal{S}_{\tau}\times\mathcal{S}_{\tau})=\rho_{\tau}^{-1}(\mathcal{S}_{\tau})\quad\text{and}\quad\overline{\mathcal{Z}_{\g}^{\tau}}=\overline{\rho}^{-1}(\mathcal{S}_{\tau}\times\mathcal{S}_{\tau})=\overline{\rho}_{\tau}^{-1}(\mathcal{S}_{\tau}).$$

This present section combines with Subsection \ref{Subsection: KW} to yield the following informal comparisons between $(\mathcal{Z}_{\mathfrak{g}}^{\tau},\overline{\mathcal{Z}_{\g}^{\tau}})$ and $(G\times\mathcal{S}_{\tau},\overline{G\times\mathcal{S}_{\tau}})$:     
\begin{itemize}
	\item $\overline{q_{\tau}}$ (resp. $\overline{\rho}_{\tau}$) is a fibrewise compactification of $\pi_{\tau}$ (resp. $q_{\tau}$);
	\item \eqref{Equation: Non-equivariant diagram} is obtained by pulling \eqref{Equation: Equivariant diagram} back along the inclusion $\mathcal{S}_{\tau}\hookrightarrow\g$;
	\item $\mathcal{Z}_{\g}^{\tau}$ and $G\times\mathcal{S}_{\tau}$ are symplectic;
	\item $\overline{\mathcal{Z}_{\g}^{\tau}}$ and $\overline{G\times\mathcal{S}_{\tau}}$ are log symplectic;
	\item $\tilde{\varphi}$ restricts to a symplectomorphism from $\mathcal{Z}_{\g}^{\tau}$ (resp. $G\times\mathcal{S}_{\tau}$) to the unique open dense symplectic leaf in $\overline{\mathcal{Z}_{\g}^{\tau}}$ (resp. $\overline{G\times\mathcal{S}_{\tau}}$).
\end{itemize}

\section{The geometries of $\overline{X}$ and $\overline{X}_{\tau}$}\label{Section: Xbar}
This section is concerned with constructing partial compactifications of Poisson slices, an issue motivated in the introduction of our paper. Our approach is to replace a Poisson slice $X_{\tau}$ with a slightly larger variety $\overline{X}_{\tau}$, provided that the latter makes sense. If $\overline{X}_{\tau}$ is well-defined, we show it to enjoy certain Poisson-geometric features and discuss the extent to which it partially compactifies $X_{\tau}$.    

\subsection{Definitions and first properties}\label{Subsection: Definition and first properties}
Fix a Hamiltonian $G$-variety $(X,P,\nu)$ and an $\mathfrak{sl}_2$-triple $\tau$ in $\g$.  The product Hamiltonian $G$-varieties $X\times (G\times\mathcal{S}_{\tau})$ and $X\times(\overline{G\times\mathcal{S}_{\tau}})$ then have respective moment maps 
$$\mu_{\tau}:X\times (G\times\mathcal{S}_{\tau})\longrightarrow\g,\quad (x,(g,y))\longrightarrow\nu(x)-\mathrm{Ad}_g(y)$$ and
$$\overline{\mu}_{\tau}:X\times (\overline{G\times\mathcal{S}_{\tau}})\longrightarrow\g,\quad (x,(\gamma,(y_1,y_2)))\longrightarrow\nu(x)-y_1.$$ 
Note also that taking the product of $$\tilde{\varphi}\big\vert_{G\times\mathcal{S}_{\tau}}:G\times\mathcal{S}_{\tau}\longrightarrow\overline{G\times\mathcal{S}_{\tau}}$$ with the identity $X\longrightarrow X$ produces a $G$-equivariant open Poisson embedding
\begin{align}\label{Equation: i_tau}
i_{\tau}:X \times (G\times\mathcal{S}_{\tau}) &\longrightarrow X\times (\overline{G\times\mathcal{S}_{\tau}})\\\nonumber
(x,(g,y)) &\longrightarrow (x,((g,e)\cdot\g_{\Delta},(\mathrm{Ad}_g(y),y))).
\end{align}
(see Proposition \ref{Proposition: nice}). One readily verifies that the diagram
\begin{equation}\begin{tikzcd}\label{Equation: Slice diag}
X\times(G\times\mathcal{S}_{\tau}) \arrow{rr}{i_{\tau}} \arrow[swap]{dr}{\mu_{\tau}}& & X\times(\overline{G\times\mathcal{S}_{\tau}}) \arrow{dl}{\overline{\mu}_{\tau}}\\
& \mathfrak{g} & 
\end{tikzcd}\end{equation}
commutes.

Now recall the Hamiltonian ($G\times G$)-variety $X\times T^*\overline{G}(\log D)$
and moment map
$$\overline{\mu}=(\overline{\mu}_L,\overline{\mu}_R): X\times T^*\overline{G}(\log D)\longrightarrow\g\oplus\g$$ from Subsection \ref{Subsection: MomentLogCotangent}. Let us write
$$\overline{X}:=(X\times T^*\overline{G}(\log D))\sslash G_L\quad\text{and}\quad\overline{X}_{\tau}:=(X\times (\overline{G\times\mathcal{S}}))\sslash G,$$
and understand ``$\overline{X}$ exists" (resp. ``$\overline{X}_{\tau}$ exists") to mean that $(X\times T^*\overline{G}(\log D))\sslash G_L$ (resp. $(X\times (\overline{G\times\mathcal{S}}))\sslash G$) exists as a geometric quotient. 

\begin{rem}
If $\tau=0$, then $X\times (\overline{G\times\mathcal{S}_{\tau}})=X\times T^*\overline{G}(\log D)$, $\overline{\mu}_{\tau}=\overline{\mu}_L$, and the $G$-action on $X\times (\overline{G\times\mathcal{S}_{\tau}})$ is the $G_L$-action $X\times T^*\overline{G}(\log D)$. One immediate consequence is that $\overline{X}=\overline{X}_0$.
\end{rem}

\begin{rem}
The action of $G_R$ on $X\times T^*\overline{G}(\log D)$ induces a residual $G$-action on $\overline{X}$, provided that $\overline{X}$ exists. This $G$-action features prominently in what follows.
\end{rem}

It is reasonable to seek conditions under which $\overline{X}$ and $\overline{X}_{\tau}$ exist. We defer this matter to Section \ref{Section: Examples}, which is largely devoted to examples. In the interim, we assume that $\overline{X}_{\tau}$ exists. Let us also recall the map $i:X\times T^*G\longrightarrow X\times T^*\overline{G}(\log D)$ from \eqref{Equation: Equivariant embedding}. This map restricts to a $G$-equivariant open embedding \begin{equation}\label{Equation: Restricted}i\big\vert_{\mu_\tau^{-1}(0)}:\mu_\tau^{-1}(0)\hookrightarrow\overline{\mu}_\tau^{-1}(0),\end{equation} which in turn descends to a morphism
\begin{equation}\label{Equation: Descent} j_\tau:(X\times (G\times\mathcal{S}_{\tau}))\sslash G\longrightarrow\overline{X}_{\tau}.\end{equation} Let us consider the composition 
\begin{equation}\label{Equation: Definition of ktau} k_{\tau}:=j_{\tau}\circ\psi_{\tau}:X_{\tau}\longrightarrow\overline{X}_{\tau},\end{equation} where $\psi_{\tau}:X_{\tau}\longrightarrow (X\times (G\times\mathcal{S}_{\tau}))\sslash G$ is the Poisson variety isomorphism from \eqref{Equation: Ham iso}. It is straightforward to verify that
\begin{equation}\label{Equation: Formula2}k_{\tau}(x)=[x:(\mathfrak{g}_{\Delta},(\nu(x),\nu(x)))]\end{equation} for all $x\in X_{\tau}$

\begin{prop}\label{Proposition: Equivariant open}
Let $\tau$ be an $\mathfrak{sl}_2$-triple in $\g$. If $\overline{X}_{\tau}$ exists, then $k_{\tau}:X_{\tau}\longrightarrow\overline{X}_{\tau}$ is an open embedding.
\end{prop}

\begin{proof}
Since $\psi_{\tau}$ is a variety isomorphism, it suffices to prove that $j_{\tau}$ is an open embedding. We achieve this by first considering the commutative square \begin{equation}\label{Equation: Square}\begin{tikzcd}
\mu_\tau^{-1}(0) \arrow[r, "i\big\vert_{\mu_\tau^{-1}(0)}"] \arrow[d]
& \overline{\mu}_\tau^{-1}(0) \arrow[d] \\
(X\times (G\times\mathcal{S}_{\tau}))\sslash G \arrow[r, "j_\tau"]
& \overline{X}_{\tau}
\end{tikzcd}.\end{equation}
The vertical morphisms are open maps by virtue of being geometric quotients \cite[Lemma 25.3.2]{Tauvel}, and we have explained that the upper horizontal map is open. It follows that $j_{\tau}$ is also an open map. Together with the observation that $j_{\tau}$ is injective, this implies that $j_{\tau}$ is an open embedding. Our proof is complete.
\end{proof}

The inclusion $X_{\tau}\longrightarrow X$ composes with the quotient map $X\longrightarrow X/G$ to yield
\begin{equation}\label{Equation: First quotient}\pi_{\tau}:X_{\tau}\longrightarrow X/G,\end{equation} provided that $X/G$ exists. We may also consider the 
morphism
\begin{equation}\label{Equation: Second quotient}\overline{\pi}_{\tau}:\overline{X}_{\tau}\longrightarrow X/G,\quad [x:(\gamma,(y_1,y_2))]\longrightarrow [x]\end{equation} if both $\overline{X}_{\tau}$ and $X/G$ exist. The following is then an immediate consequence of \eqref{Equation: Formula2}.

\begin{prop}\label{Proposition: X quotient commute}
Let $\tau$ be an $\mathfrak{sl}_2$-triple in $\g$. If $\overline{X}_{\tau}$ and $X/G$ exist, then the diagram
\begin{equation}\begin{tikzcd}\label{Equation: Partial compactification}
X_{\tau} \arrow{rr}{k_{\tau}} \arrow[swap]{dr}{\pi_{\tau}}& & \overline{X}_{\tau} \arrow{dl}{\overline{\pi}_{\tau}}\\
& X/G & 
\end{tikzcd}\end{equation}
commutes.
\end{prop}

This diagram is particularly noteworthy if $\tau$ is a principal $\mathfrak{sl}_2$-triple.

\begin{thm}\label{Theorem: Compact}
Let $\tau$ be a principal $\mathfrak{sl}_2$-triple in $\g$. If $\overline{X}_{\tau}$ and $X/G$ exist, then the diagram \eqref{Equation: Partial compactification} realizes $\overline{\pi}_{\tau}$ as a fibrewise compactification of $\pi_{\tau}$. 
\end{thm}

\begin{proof}
Our objective is to prove that $\overline{\pi}_{\tau}$ has projective fibres. Let us begin by fixing a point $x\in X$. We then have
\begin{equation}\label{Equation: Statement}\overline{\pi}_{\tau}^{-1}([x])=\{[x:(\gamma,(\nu(x),y))]:\gamma\in\overline{G},\text{ }y\in\mathcal{S}_{\tau},\text{ and }(\nu(x),y)\in\gamma\}.\end{equation} On the other hand, it is known that $y_1,y_2\in\g$ belong to the same fibre of the adjoint quotient $\chi:\g\longrightarrow\mathrm{Spec}(\mathbb{C}[\g]^G)$ whenever $(y_1,y_2)\in\gamma$ for some $\gamma\in\overline{G}$ (see Subsection \ref{Subsection: KW}). The discussion and notation in Subsection \ref{Subsection: Lie theoretic conventions} associated with principal $\mathfrak{sl}_2$-triples then imply the following: if $y_1\in\g$ and $y_2\in\mathcal{S}_{\tau}$ are such that $(y_1,y_2)\in\gamma$ for some $\gamma\in\overline{G}$, then $y_2=(y_1)_{\tau}$. We may therefore present \eqref{Equation: Statement} as the statement
$$\overline{\pi}_{\tau}^{-1}([x])=\{[x:(\gamma,(\nu(x),\nu(x)_{\tau}))]:\gamma\in\overline{G}\text{ and }(\nu(x),\nu(x)_{\tau})\in\gamma\}.$$ In other words, $\overline{\pi}_{\tau}^{-1}([x])$ is the image of the closed subvariety
$$\{\gamma\in\overline{G}:(\nu(x),\nu(x)_{\tau})\in\gamma\}\subseteq\overline{G}$$ under the morphism
$$\overline{G}\longrightarrow \overline{X}_{\tau},\quad \gamma\longrightarrow [x:(\gamma,(\nu(x),\nu(x)_{\tau}))].$$ This subvariety is projective by virtue of being closed in $\overline{G}$, and we conclude that $\overline{\pi}_{\tau}^{-1}([x])$ is projective. This completes the proof.
\end{proof}	

Let us also examine the case $\tau=0$ in some detail. To this end, assume that $\overline{X}_0=\overline{X}$ exists and consider the geometric quotient map
$$\overline{\pi}_L:\overline{\mu}_L^{-1}(0)\longrightarrow\overline{X}.$$ 
The $G_R$-action on $\overline{\mu}_L^{-1}(0)$ then descends under $\overline{\pi}_L$ to a $G$-action $\overline{X}$. On the other hand, note that the restriction of $$-\overline{\mu}_R:X\times T^*\overline{G}(\log D)\longrightarrow\g,\quad (x,(\gamma,(y_1,y_2)))\longrightarrow y_2$$ to $\overline{\mu}_L^{-1}(0)$ is $G_R$-equivariant and $G_L$-invariant. This restriction therefore descends under $\overline{\pi}_L$ to the $G$-equivariant morphism
\begin{equation}\label{Equation: Formula1}\overline{\nu}:\overline{X}\longrightarrow\g,\quad [x:(\gamma,(\nu(x),y))]\longrightarrow y.\end{equation} Let us write $k:X\longrightarrow\overline{X}$, $\pi:X\longrightarrow X/G$, and $\overline{\pi}:\overline{X}\longrightarrow X/G$ for \eqref{Equation: Definition of ktau}, \eqref{Equation: First quotient}, and \eqref{Equation: Second quotient}, respectively, in the case $\tau=0$. 

\begin{prop}\label{Proposition: Equivariant}
If $\overline{X}$ exists, then $k:X\longrightarrow\overline{X}$ is a $G$-equivariant open embedding and 
\begin{equation}\label{Equation: Triangle}\begin{tikzcd}
X \arrow{rr}{k} \arrow[swap]{dr}{\nu}& & \overline{X} \arrow{dl}{\overline{\nu}}\\
& \mathfrak{g} & 
\end{tikzcd}\end{equation}
commutes. If $X/G$ also exists, then 
\begin{equation}\label{Equation: Same diag}\begin{tikzcd}
X \arrow{rr}{k} \arrow[swap]{dr}{\pi}& & \overline{X} \arrow{dl}{\overline{\pi}}\\
& X/G & 
\end{tikzcd}\end{equation}
commutes.
\end{prop}

\begin{proof}
The commutativity of \eqref{Equation: Triangle} follows immediately from \eqref{Equation: Formula1} and \eqref{Equation: Formula2}, while Proposition \ref{Proposition: X quotient commute} forces \eqref{Equation: Same diag} to commute. Proposition \ref{Proposition: Equivariant open} implies that $k$ is an open embedding. Our equivariance claim follows from \eqref{Equation: Formula2}, the above-given definition of the $G$-action on $\overline{X}$, and a direct calculation. This completes the proof. 
\end{proof}

\subsection{The Poisson geometries of $\overline{X}$ and $\overline{X}_{\tau}$}\label{Subsection: Poisson geometries tau}
Let $(X,P,\nu)$ be a Hamiltonian $G$-variety and suppose that $\tau$ is an $\mathfrak{sl}_2$-triple in $\g$. In what follows, we show that the Poisson slice $X_{\tau}$ endows $\overline{X}_{\tau}$ with certain Poisson-geometric qualities. The most basic such feature is as follows. 

\begin{prop}
	Let $\tau$ be an $\mathfrak{sl}_2$-triple in $\g$. If $\overline{X}_{\tau}$ exists, then $\mathbb{C}[\overline{X}_{\tau}]$ carries a natural Poisson bracket for which $k_{\tau}^*:\mathbb{C}[\overline{X}_{\tau}]\longrightarrow\mathbb{C}[X_{\tau}]$ is a Poisson algebra morphism.
\end{prop}

\begin{proof}
	The definition 
	$$\overline{X}_{\tau}:=(X\times(\overline{G\times\mathcal{S}_{\tau}}))\sslash G$$
	combines with the discussion in Subsection \ref{Subsection: Hamiltonian reduction} to yield a Poisson bracket on $\mathbb{C}[\overline{X}_{\tau}]$, as well as the following facts:
	\begin{itemize}
		\item[(i)] $\mathbb{C}[X\times (\overline{G\times\mathcal{S}_{\tau}})]^{G}$ is a Poisson subalgebra of $\mathbb{C}[X\times (\overline{G\times\mathcal{S}_{\tau}})]$;
		\item[(ii)] $\mathbb{C}[\overline{\mu}_{\tau}^{-1}(0)]^{G}$ has a unique Poisson bracket for which restriction $$\beta:\mathbb{C}[X\times (\overline{G\times\mathcal{S}_{\tau}})]^{G}\longrightarrow \mathbb{C}[\overline{\mu}_{\tau}^{-1}(0)]^{G}$$ is a Poisson algebra morphism;
		\item[(iii)] the geometric quotient map $\overline{\mu}_\tau^{-1}(0)\longrightarrow\overline{X}_{\tau}$ induces a Poisson algebra isomorphism
		$$\delta:\mathbb{C}[\overline{X}_{\tau}]\overset{\cong}\longrightarrow \mathbb{C}[\overline{\mu}_{\tau}^{-1}(0)]^{G}.$$
	\end{itemize}
	
	We have a row of Poisson algebra morphisms    
	$$\mathbb{C}[X\times (\overline{G\times\mathcal{S}_{\tau}})]\overset{\alpha}\longleftarrow \mathbb{C}[X\times (\overline{G\times\mathcal{S}_{\tau}})]^{G}\overset{\beta}\longrightarrow \mathbb{C}[\overline{\mu}_\tau^{-1}(0)]^{G}\overset{\delta}\longleftarrow\mathbb{C}[\overline{X}_{\tau}],$$ where $\alpha$ is the inclusion. An analogous procedure yields a second row
	$$\mathbb{C}[X\times (G\times\mathcal{S}_{\tau})]\overset{\alpha'}\longleftarrow \mathbb{C}[X\times(G\times\mathcal{S}_{\tau})]^{G}\overset{\beta'}\longrightarrow \mathbb{C}[\mu_{\tau}^{-1}(0)]^{G}\overset{\delta'}\longleftarrow\mathbb{C}[X_{\tau}]$$
	of Poisson algebra morphisms. Now recall the $G$-equivariant open Poisson embedding 
	$$i_{\tau}:X\times (G\times\mathcal{S}_{\tau})\longrightarrow X\times (\overline{G\times\mathcal{S}_{\tau}})$$ from \eqref{Equation: i_tau}, as well as the commutative diagram \eqref{Equation: Slice diag}. It follows that $i_{\tau}$ induces the first three vertical arrows in the commutative diagram
	$$\begin{tikzcd}
	\mathbb{C}[X\times (\overline{G\times\mathcal{S}_{\tau}})]  \arrow[d, "\alpha''"]
	& \arrow[l, swap, "\alpha"] \mathbb{C}[X\times (\overline{G\times\mathcal{S}_{\tau}})]^{G} \arrow[d, "\beta''"] \arrow[r, "\beta"] & \mathbb{C}[\overline{\mu}_{\tau}^{-1}(0)]^{G} \arrow[d, "\delta''"] & \mathbb{C}[\overline{X}_{\tau}] \arrow[d, "k_{\tau}^*"] \arrow[l, swap, "\delta"] \\
	\mathbb{C}[X\times (G\times\mathcal{S}_{\tau})] 
	& \arrow[l, swap, "\alpha'"] \mathbb{C}[X\times (G\times\mathcal{S}_{\tau})]^{G} \arrow[r, "\beta'"] & \mathbb{C}[\mu_{\tau}^{-1}(0)]^{G}  & \mathbb{C}[X_{\tau}] \arrow[l, swap, "\delta'"]
	\end{tikzcd}.$$
	Observe that $\alpha''$ is a Poisson algebra morphism, as follows from $i_{\tau}$ being a Poisson morphism. One deduces that $\beta''$ must also be a Poisson algebra morphism. This combines with the commutativity of the middle square and the fact that $\beta'$ and $\beta''$ are surjective Poisson algebra morphisms to imply that $\delta''$ is a Poisson algebra morphism. Since $\delta$ and $\delta'$ are Poisson algebra isomorphisms, this forces $k_{\tau}^*$ to be a Poisson algebra morphism. 
\end{proof}

Some more manifestly geometric features of $\overline{X}_{\tau}$ may be developed as follows. Write $(X\times (\overline{G\times\mathcal{S}_{\tau}}))^{\circ}$ for the $G$-invariant open subvariety of points in $X\times (\overline{G\times\mathcal{S}_{\tau}})$ whose $G$-stabilizers are trivial. 

The $G$-action on $(X\times \overline{G\times\mathcal S_\tau})^{\circ}$ is Hamiltonian with respect to the Poisson structure that $(X\times \overline{G\times \mathcal S_\tau})^{\circ}$ inherits from $(X\times (\overline{G\times\mathcal{S}_{\tau}}))$, and  
\begin{equation}\label{Equation: Complicated}\overline{\mu}_\tau^{\circ}:=\overline{\mu}_\tau\bigg\vert_{(X\times (\overline{G\times\mathcal{S}_{\tau}}))^{\circ}}:(X\times (\overline{G\times\mathcal{S}_{\tau}}))^{\circ}\longrightarrow\mathfrak{g}
\end{equation} 
is a moment map. 

Now assume that $\overline{X}_{\tau}$ exists and consider the geometric quotient map
$$\overline{\theta}_{\tau}:\overline{\mu}_{\tau}^{-1}(0)\longrightarrow\overline{X}_{\tau}.$$ The variety $(\overline{\mu}_{\tau}^{\circ})^{-1}(0)$ is $G$-invariant and open in $\mu_{\tau}^{-1}(0)$, and we set
$$X_{\tau}^{\circ}:=\overline{\theta}_{\tau}((\overline{\mu}_{\tau}^{\circ})^{-1}(0))\subseteq\overline{X}_{\tau}.$$
We also let
$$\overline{\theta}_{\tau}^{\circ}:(\overline{\mu}_{\tau}^{\circ})^{-1}(0)\longrightarrow\overline{X}_{\tau}^{\circ}$$ denote the restriction of $\overline{\theta}_\tau$ to $(\overline{\mu}_\tau^{\circ})^{-1}(0)$.

\begin{lem}\label{Lemma: Existence}
Let $\tau$ be an $\mathfrak{sl}_2$-triple in $\g$. If $\overline{X}_{\tau}$ exists, then $\overline{X}_{\tau}^{\circ}$ is an open subvariety of $\overline{X}_{\tau}$ and $\overline{\theta}^{\circ}_\tau:(\overline{\mu}_\tau^{\circ})^{-1}(0)\longrightarrow\overline{X}_{\tau}^{\circ}$ is the geometric quotient of $(\overline{\mu}_\tau^{\circ})^{-1}(0)$ by $G$.
\end{lem}

\begin{proof}
The geometric quotient map $\overline{\theta}_\tau:\overline{\mu}_\tau^{-1}(0)\longrightarrow\overline{X}_\tau$ is open \cite[Lemma 25.3.2]{Tauvel}. It follows that $\overline{X}_{\tau}^{\circ}=\overline{\theta}_\tau((\overline{\mu}_\tau^{\circ})^{-1}(0))$ is an open subvariety of $\overline{X}_{\tau}$. The rest of this lemma is an immediate consequence of Proposition \ref{Proposition: Restriction of quotient}.
\end{proof}

\begin{prop}\label{Proposition: Poisson structure}
Let $\tau$ be an $\mathfrak{sl}_2$-triple in $\g$. If $\overline{X}_{\tau}$ exists, then $\overline{X}_{\tau}^{\circ}$ is smooth and Poisson.
\end{prop}

\begin{proof}
Recall that $(X\times (\overline{G\times\mathcal{S}_{\tau}}))^{\circ}$ is a Hamiltonian $G$-variety with moment map \eqref{Equation: Complicated}. Lemma \ref{Lemma: Existence} then implies that $\overline{X}_{\tau}^{\circ}$ is the Hamiltonian reduction of $(X\times (\overline{G\times\mathcal{S}_{\tau}}))^{\circ}$ at level zero.
The proposition now follows from generalities about Hamiltonian reductions by free actions, the relevant parts of which are discussed in Subsection \ref{Subsection: Hamiltonian reduction}. 
\end{proof}

Now recall the open embedding $k_{\tau}:X_{\tau}\longrightarrow\overline{X}_{\tau}$ defined in \eqref{Equation: Definition of ktau}.

\begin{prop}\label{Proposition: Open embedding}
Let $\tau$ be an $\mathfrak{sl}_2$-triple in $\g$, and assume that $\overline{X}_{\tau}$ exists. The image of $k_{\tau}:X_{\tau}\longrightarrow\overline{X}_{\tau}$ then lies in $\overline{X}_{\tau}^{\circ}$, and $k_{\tau}$ defines an open embedding of Poisson varieties $X_{\tau}\longrightarrow\overline{X}_{\tau}^{\circ}$.
\end{prop}

\begin{proof}
Recall that $k_{\tau}=j_{\tau}\circ\psi_{\tau}$, and that $\psi_{\tau}$ is a Poisson variety isomorphism. It therefore suffices to prove the following: 
\begin{itemize}
\item[(i)] the image of $j_\tau:(X\times (G\times\mathcal{S}_{\tau}))\sslash G\longrightarrow\overline{X}_{\tau}$ lies in $\overline{X}_{\tau}^{\circ}$; 
\item[(ii)] $j_\tau$ defines an open embedding of Poisson varieties $(X\times (G\times\mathcal{S}_{\tau}))\sslash G\longrightarrow\overline{X}_{\tau}^{\circ}$.
\end{itemize}
	
Since $G$ acts freely on $X\times (G\times\mathcal{S}_\tau)$, the image of \eqref{Equation: i_tau} lies in $(X\times (\overline{G\times\mathcal{S}_{\tau}}))^{\circ}$. We may therefore interpret \eqref{Equation: i_tau} as a $G$-equivariant open Poisson embedding 
$$i_{\tau}:X\times ( G\times\mathcal{S}_{\tau})\longrightarrow (X\times (\overline{G\times\mathcal{S}_{\tau}}))^{\circ}$$
and \eqref{Equation: Slice diag} as a commutative diagram
$$\begin{tikzcd}
X\times (G\times\mathcal{S}_{\tau}) \arrow{rr}{i_{\tau}} \arrow[swap]{dr}{\mu_{\tau}}& & (X\times (\overline{G\times\mathcal{S}_{\tau}}))^{\circ} \arrow{dl}{\overline{\mu}_{\tau}^{\circ}}\\
& \mathfrak{g} & 
\end{tikzcd}.$$
Such considerations allow one to regard \eqref{Equation: Restricted} and \eqref{Equation: Descent} as maps \begin{equation}\label{Equation: New Restricted}i_{\tau}\big\vert_{\mu_\tau^{-1}(0)}:\mu_\tau^{-1}(0)\hookrightarrow(\overline{\mu}_\tau^{\circ})^{-1}(0)\end{equation} and
\begin{equation}\label{Equation: New j_tau} j_{\tau}:(X\times (G\times\mathcal{S}_{\tau}))\sslash G\longrightarrow\overline{X}_{\tau}^{\circ},\end{equation} respectively. This verifies (i) and yields the commutative square
\begin{equation}\label{Equation: New Square}\begin{tikzcd}
\mu_\tau^{-1}(0) \arrow[r, "i_\tau\big\vert_{\mu_\tau^{-1}(0)}"] \arrow[d]
& (\overline{\mu}_\tau^{\circ})^{-1}(0) \arrow[d] \\
(X\times (G\times\mathcal{S}_{\tau}))\sslash G \arrow[r, "j_\tau"]
& \overline{X}_{\tau}^{\circ}
\end{tikzcd}.\end{equation}
By combining this square with the description of the Poisson structure on a Hamiltonian reduction, we deduce that \eqref{Equation: New j_tau} is a Poisson morphism. This morphism is also an open embedding, as follows easily from Proposition \ref{Proposition: Equivariant open}. Our proof is therefore complete.
\end{proof}

Let us write $\overline{X}^{\circ}$ for $\overline{X}_\tau^{\circ}$ if $\tau=0$. This variety turns out to enjoy some Poisson geometric features beyond those of a general $\overline{X}_{\tau}^{\circ}$. To develop these features, assume that $\overline{X}$ exists and let
$$\overline{\pi}_L:\overline{\mu}_L^{-1}(0)\longrightarrow\overline{X}$$
be the geometric quotient map. Write $(X\times T^*\overline{G}(\log D))^{\circ}$ for the ($G\times G$)-invariant open subvariety of points in $X\times T^*\overline{G}(\log D)$ whose $G_L$-stabilizers are trivial. The ($G\times G$)-action on $(X\times T^*\overline{G}(\log D))^{\circ}$ is Hamiltonian with respect to the Poisson structure that $(X\times T^*\overline{G}(\log D))^{\circ}$ inherits from $X\times T^*\overline{G}(\log D)$, and  \begin{equation}\label{Equation: Complicated2}(\overline{\mu}_L^{\circ},\overline{\mu}_R^{\circ}):=(\overline{\mu}_L\bigg\vert_{(X\times T^*\overline{G}(\log D))^{\circ}},\overline{\mu}_R\bigg\vert_{(X\times T^*\overline{G}(\log D))^{\circ}}):(X\times T^*\overline{G}(\log D))^{\circ}\longrightarrow\mathfrak{g}\oplus\mathfrak{g}\end{equation} 
is a moment map. 

Now consider the ($G\times G$)-invariant open subvariety of $(\overline{\mu}_L^{\circ})^{-1}(0)$ of $\overline{\mu}_L^{-1}(0)$, and observe that
$$\overline{X}^{\circ}:=\overline{\pi}_L((\overline{\mu}_L^{\circ})^{-1}(0)).$$ Let $$\overline{\pi}^{\circ}_L:(\overline{\mu}_L^{\circ})^{-1}(0)\longrightarrow\overline{X}^{\circ}$$ denote the restriction of $\overline{\pi}_L$ to $(\overline{\mu}_L^{\circ})^{-1}(0)$. At the same time, recall the definition of the $G$-action on $\overline{X}$.

\begin{lem}
Assume that $\overline{X}$ exists. The subset $\overline{X}^{\circ}$ is then a $G$-invariant open subvariety of $\overline{X}$, and $\overline{\pi}^{\circ}_L:(\overline{\mu}_L^{\circ})^{-1}(0)\longrightarrow\overline{X}^{\circ}$ is the geometric quotient of $(\overline{\mu}_L^{\circ})^{-1}(0)$ by $G_L$. 
\end{lem}

\begin{proof}
Observe that $\overline{\pi}_L$ is equivariant with respect to the action of $G_R$ on $\overline{\mu}_L^{-1}(0)$ and the above-discussed $G$-action on $\overline{X}$. Since $(\overline{\mu}_L^{\circ})^{-1}(0)$ is $G_R$-invariant in $\overline{\mu}_L^{-1}(0)$, this implies that $\overline{X}^{\circ}=\overline{\pi}_L((\overline{\mu}_L^{\circ})^{-1}(0))$ is $G$-invariant in $\overline{X}$. The rest of this lemma is an immediate consequence of Lemma \ref{Lemma: Existence}.
\end{proof}

The $G$-action that $\overline{X}^{\circ}$ inherits from $\overline{X}$ is compatible with the Poisson variety structure referenced in Proposition \ref{Proposition: Poisson structure}. To formulate this more precisely, recall the map $\nu:\overline{X}\longrightarrow\g$ in \eqref{Equation: Formula1} and set
$$\overline{\nu}^{\circ}:=\overline{\nu}\big\vert_{\overline{X}^{\circ}}:\overline{X}^{\circ}\longrightarrow\g.$$ 

\begin{prop}
If $\overline{X}$ exists, then the action of $G$ on $\overline{X}^{\circ}$ is Hamiltonian with moment map $\overline{\nu}^{\circ}:\overline{X}^{\circ}\longrightarrow\g$.
\end{prop}

\begin{proof}
Recall that $(X\times T^*\overline{G}(\log D))^{\circ}$ is a Hamiltonian ($G\times G$)-variety with moment map \eqref{Equation: Complicated2}. One deduces that $\overline{X}^{\circ}=(\overline{\mu}_L^{\circ})^{-1}(0)/G_L$ is a Hamiltonian $G$-variety, and that the corresponding moment map is obtained by letting $$-\overline{\mu}_R^{\circ}\bigg\vert_{(\overline{\mu}_L^{\circ})^{-1}(0)}: (\overline{\mu}_L^{\circ})^{-1}(0)\longrightarrow\g$$
descend to $\overline{X}^{\circ}$. It remains only to observe that this descended moment map and the $G$-action on $\overline{X}^{\circ}$ are restrictions of $\overline{\nu}:\overline{X}\longrightarrow\g$ and the $G$-action on $\overline{X}$, respectively.
\end{proof}

\begin{prop}\label{Proposition: Easy}
Assume that $\overline{X}$ exists. The image of $k:X\longrightarrow\overline{X}$ then lies in $\overline{X}^{\circ}$, and $k$ defines an open embedding of Hamiltonian $G$-varieties $X\longrightarrow\overline{X}^{\circ}$.
\end{prop}

\begin{proof}
This is a direct consequence of Propositions \ref{Proposition: Equivariant} and Proposition \ref{Proposition: Open embedding}.
\end{proof}

\subsection{The log symplectic geometries of $\overline{X}$ and $\overline{X}_{\tau}$}
We now examine the Poisson geometries of $\overline{X}$ and $\overline{X}_{\tau}$ in the special case of a symplectic Hamiltonian $G$-variety $(X,P,\nu)$. These Poisson geometries essentially become log symplectic geometries, as is consistent with the following result. Recall the map $i_{\tau}:X\times (G\times\mathcal{S}_{\tau})\longrightarrow X\times (\overline{G\times\mathcal{S}_{\tau}})$ defined in \eqref{Equation: i_tau}.

\begin{lem}
Let $(X,P,\nu)$ be an irreducible symplectic Hamiltonian $G$-variety. If $\tau$ is an $\mathfrak{sl}_2$-triple in $\g$, then the following statements then hold:
\begin{itemize}
\item[(i)] $X\times (\overline{G\times\mathcal{S}_{\tau}})$ is log symplectic; 
\item[(ii)] $i_{\tau}$ is a $G$-equivariant symplectomorphism onto the unique open dense symplectic leaf in $X\times (\overline{G\times\mathcal{S}_{\tau}})$.	
\end{itemize}
\end{lem}

\begin{proof}
Proposition \ref{Proposition: nice} tells us that  $$\tilde{\varphi}\big\vert_{G\times\mathcal{S}_{\tau}}:G\times\mathcal{S}_{\tau}\longrightarrow\overline{G\times\mathcal{S}_{\tau}}$$ is a $G$-equivariant symplectomorphism onto the open dense symplectic leaf in the log symplectic variety $\overline{G\times\mathcal{S}_{\tau}}$. We also recall that $i_{\tau}$ is the product of $\tilde{\varphi}\big\vert_{G\times\mathcal{S}_{\tau}}$ with the identity $X\longrightarrow X$. These last two sentences imply that $i_{\tau}$ is a $G$-equivariant symplectomorphism onto the complement of the degeneracy locus in $X\times(\overline{G\times\mathcal{S}_{\tau}})$. Since $X\times (G\times\mathcal{S}_{\tau})$ and $X\times (\overline{G\times\mathcal{S}_{\tau}})$ are irreducible, this implies that the image of $i_{\tau}$ is the unique open dense symplectic leaf in $X\times(\overline{G\times\mathcal{S}_{\tau}})$.

Now consider the closed subvariety
\begin{equation}\label{Equation: Complement}D_{\tau}:=(X\times(\overline{G\times\mathcal{S}_{\tau}}))\setminus i_{\tau}(X\times (G\times\mathcal{S}_{\tau}))\end{equation} of $X\times(\overline{G\times\mathcal{S}_{\tau}})$ It remains only to prove the following things:
\begin{itemize}
\item[(a)] $D_{\tau}$ is a normal crossing divisor;
\item[(b)] the top exterior power of the Poisson bivector on $X\times(\overline{G\times\mathcal{S}_{\tau}})$ has a reduced vanishing locus;
\item[(c)] the vanishing locus in (b) coincides with $D_{\tau}$.
\end{itemize}
 To this end, we note that
$$D_{\tau}=X\times(\overline{G\times\mathcal{S}_{\tau}}\setminus\tilde{\varphi}(G\times\mathcal{S}_{\tau})).$$ We also observe that $\overline{G\times\mathcal{S}_{\tau}}\setminus\tilde{\varphi}(G\times\mathcal{S}_{\tau})$ is a normal crossing divisor in $\overline{G\times\mathcal{S}_{\tau}}$, as
$\tilde{\varphi}(G\times\mathcal{S}_{\tau})$ is the unique open dense symplectic leaf in the log symplectic variety $\overline{G\times\mathcal{S}_{\tau}}$ (see Proposition \ref{Proposition: nice}). The previous two sentences then force $D_{\tau}$ to be a normal crossing divisor in $\overline{G\times\mathcal{S}_{\tau}}$, i.e. (a) holds. The assertion (b) follows immediately from $X$ being symplectic and $\overline{G\times\mathcal{S}_{\tau}}$ being log symplectic. The assertion (c) follows from our description of the degeneracy locus in $X\times(\overline{G\times\mathcal{S}_{\tau}})$, as provided in the first paragraph of the proof. Our proof is therefore complete.
\end{proof}

Now recall the open embedding $k_{\tau}:X_{\tau}\longrightarrow\overline{X}_{\tau}$ in \eqref{Equation: Definition of ktau}, as well as the fact that $k_{\tau}(X_{\tau})\subseteq\overline{X}_{\tau}^{\circ}$ (see Proposition \ref{Proposition: Open embedding}). If $X_{\tau}$ is irreducible, then $k_{\tau}(X_{\tau})$ lies in a unique irreducible component $(\overline{X}_{\tau}^{\circ})_{\text{irr}}$ of the Poisson variety $\overline{X}_{\tau}^{\circ}$. The log symplectic nature of $\overline{X}_{\tau}$ is then captured by the following result, which relies heavily on the notation of Subsection \ref{Subsection: Poisson geometries tau}.

\begin{thm}\label{Theorem: nice}
Let $(X,P,\nu)$ be an irreducible symplectic Hamiltonian $G$-variety. Suppose that $\tau$ is an $\mathfrak{sl}_2$-triple in $\g$, and that $X_{\tau}$ is irreducible. If $\overline{X}_{\tau}$ exists, then the following statements hold.
\begin{itemize}
\item[(i)] The Poisson variety $(\overline{X}_{\tau}^{\circ})_{\emph{irr}}$ is log symplectic.
\item[(ii)] The morphism $k_{\tau}:X_{\tau}\longrightarrow\overline{X}_{\tau}$ is a symplectomorphism onto the unique open dense symplectic leaf in $(\overline{X}_{\tau}^{\circ})_{\emph{irr}}$. 
\end{itemize} 
\end{thm}

\begin{proof}
	Since $G$ acts freely on $(X\times (\overline{G\times\mathcal{S}_{\tau}}))^{\circ}$, the variety $(\overline{\mu}_\tau^{\circ})^{-1}(0)$ is smooth. The irreducible components of $(\overline{\mu}_\tau^{\circ})^{-1}(0)$ are therefore pairwise disjoint, while the connectedness of $G$ forces these components to be $G$-invariant. It follows that the irreducible components of $\overline{X}_\tau^{\circ}$ are precisely the images of the irreducible components of $(\overline{\mu}_\tau^{\circ})^{-1}(0)$ under the quotient map
	$$\overline{\theta}_{\tau}^{\circ}:(\overline{\mu}_L^{\circ})^{-1}(0)\longrightarrow\overline{X}_{\tau}^{\circ}.$$ This implies that $(\overline{X}_{\tau}^{\circ})_{\text{irr}}=\overline{\theta}_{\tau}^{\circ}(Y)$ for some unique irreducible component $Y\subseteq (\overline{\mu}_\tau^{\circ})^{-1}(0)$. 
	
	Now note that the image of \eqref{Equation: Restricted} lies in a unique irreducible component $Z$ of the smooth variety $(\overline{\mu}_L^{\circ})^{-1}(0)$, as $X_{\tau}$ and $\mu_{\tau}^{-1}(0)$ are irreducible. We also note that $\overline{\theta}_{\tau}^{\circ}(Z)$ contains the image of $k_{\tau}$, as follows from the commutativity of \eqref{Equation: Square}. We conclude that $\overline{\theta}_{\tau}^{\circ}(Z)=(\overline{X}^{\circ}_{\tau})_{\text{irr}}$, and the previous paragraph then implies that $Z=Y$.
	
	In light of the above, \eqref{Equation: Restricted} may be interpreted as an open embedding
	\begin{equation}\label{Equation: Map}i\big\vert_{\mu_\tau^{-1}(0)}:\mu_\tau^{-1}(0)\longrightarrow Y.\end{equation} The irreducibility of $Y$ forces the complement of the image to have positive codimension in $Y$. This complement is easily checked to be $Y\cap D_{\tau}$, where $D_{\tau}\subseteq X\times (\overline{G\times\mathcal{S}_{\tau}})$ is defined in \eqref{Equation: Complement}. We also observe that $Y\cap D_{\tau}$ has codimension at most one in $Y$, as $D_{\tau}$ is a divisor in $X\times (\overline{G\times\mathcal{S}_{\tau}})$. These last three sentences imply that $Y\cap D_{\tau}$ is a divisor in $Y$. By \cite[Proposition 3.6]{BalibanuUniversal}, the Poisson structure on $\overline{\theta}_{\tau}^{\circ}(Y)=(\overline{X}_{\tau}^{\circ})_{\text{irr}}$ is log symplectic with divisor $\overline{\theta}_{\tau}^{\circ}(Y\cap D_{\tau})$. This completes the proof of (i).
	
	Now consider the commutative diagram
	$$\begin{tikzcd}
	\mu_\tau^{-1}(0) \arrow[r, "i\big\vert_{\mu_\tau^{-1}(0)}"] \arrow[d]
	& Y \arrow[d] \\
	X_{\tau} \arrow[r, "k_\tau"]
	& (\overline{X}_{\tau}^{\circ})_{\text{irr}}
	\end{tikzcd},$$
	where the right vertical map is the restriction of $\overline{\theta}_{\tau}^{\circ}$. Since $Y\cap D_{\tau}$ is the complement of the image of \eqref{Equation: Map}, we deduce that the image of $k_\tau$ has a complement of $\overline{\theta}_{\tau}^{\circ}(Y\cap D_\tau)$. This amounts to the image of $k_{\tau}$ being the unique open dense symplectic leaf in $\overline{X}^{\circ}_{\text{irr}}$. Proposition \ref{Proposition: Open embedding} then implies that $k_\tau$ is a symplectomorphism onto this leaf. This establishes (ii), completing the proof.
\end{proof}

It is worth examining this result in the case $\tau=0$. To this end, recall the open embedding $k:X\longrightarrow\overline{X}^{\circ}$ from Subsection \ref{Subsection: Definition and first properties} and the fact that $k(X)\subseteq\overline{X}^{\circ}$ (see Proposition \ref{Proposition: Easy}). If $X$ is irreducible, then $k(X)$ lies in a unique irreducible component $\overline{X}^{\circ}_{\text{irr}}$ of $\overline{X}^{\circ}$. On the other hand, recall the $G$-actions on $\overline{X}$ and $\overline{X}^{\circ}$ discussed in Subsection \ref{Subsection: Poisson geometries tau}. Let us also recall the map $\overline{\nu}:\overline{X}\longrightarrow\g$ from \eqref{Equation: Formula1}.  

\begin{cor}\label{Corollary: LogSymplecticCompletion}
Let $(X,P,\nu)$ be an irreducible symplectic Hamiltonian $G$-variety. If $\overline{X}$ exists, then the following statements hold.
\begin{itemize}
\item[(i)] The Poisson variety $\overline{X}^{\circ}_{\emph{irr}}$ is log symplectic.
\item[(ii)] The $G$-action on $\overline{X}$ restricts to a Hamiltonian $G$-action on $\overline{X}^{\circ}_{\emph{irr}}$ with moment map
$$\overline{\nu}\bigg\vert_{\overline{X}^{\circ}_{\emph{irr}}}:\overline{X}^{\circ}_{\emph{irr}}\longrightarrow\g.$$
\item[(iii)] The morphism $k:X\longrightarrow\overline{X}$ is a $G$-equivariant symplectomorphism onto the unique open dense symplectic leaf in $\overline{X}^{\circ}_{\emph{irr}}$.
\item[(iv)] The symplectomorphism in $\mathrm{(iii)}$ is a embedding of Hamiltonian $G$-varieties. 
\end{itemize}
\end{cor}

\begin{proof}
Note that $\mu_{\tau}=\mu_L$ if $\tau=0$, where $\mu_L:X\times T^*G\longrightarrow\g$ is the moment map for the Hamiltonian action of $G_L=G\times\{e\}\subseteq G\times G$ on $T^*G$. We also observe that the map
$$X\times G\longrightarrow\mu_L^{-1}(0),\quad (x,g)\longrightarrow (x,(g,\mathrm{Ad}_{g^{-1}}(\nu(x)))),\quad (x,g)\in X\times G$$ is a variety isomorphism. It follows that $\mu_{\tau}^{-1}(0)$ is irreducible if $\tau=0$. Theorem \ref{Theorem: nice} now implies that $\overline{X}^{\circ}_{\text{irr}}$ is log symplectic, and that $k:X\longrightarrow\overline{X}$ is a symplectomorphism onto the unique open dense symplectic leaf in $\overline{X}^{\circ}_{\text{irr}}$. One also knows that $k$ defines an embedding of Hamiltonian $G$-varieties $X\longrightarrow\overline{X}^{\circ}$ (see Proposition \ref{Proposition: Easy}), and that the $G$-action on $\overline{X}^{\circ}$ must preserve the component $\overline{X}^{\circ}_{\text{irr}}$. These last two sentences serve to verify (i)--(iv).
\end{proof}

\section{Examples}\label{Section: Examples}
We now illustrate some of our results in the context of concrete and familiar examples.   

\subsection{The existence of $\overline{X}_{\tau}$ and $\overline{X}$} Our first step is to find sufficient conditions for the existence of $\overline{X}_{\tau}$ and $\overline{X}$. To this end, let $(X,P,\nu)$ be a Hamiltonian $G$-variety. Consider the product ($G\times G$)-variety $X\times\overline{G}$, where the ($G\times G$)-action on $X$ is the one described in Subsection \ref{Subsection: MomentLogCotangent}. Let $\tau$ be an $\mathfrak{sl}_2$-triple in $\g$ and consider the following conditions:
\begin{itemize}
	\item[(I)] $X$ is a principal $G$-bundle;
	\item[(II)] $(X\times\overline{G})/G_L$ exists;
	\item[(III)] $\overline{X}$ exists;
	\item[(IV)] $\overline{X}_{\tau}$ exists.
\end{itemize}

\begin{lem}\label{Lemma: Implications}
	Let $(X,P,\nu)$ be a Hamiltonian $G$-variety, and suppose that $\tau$ is an $\mathfrak{sl}_2$-triple in $\g$. We then have the chain of implications $\mathrm{(I)}\Longrightarrow\mathrm{(II)}\Longrightarrow\mathrm{(III)}\Longrightarrow\mathrm{(IV)}$.
\end{lem}

\begin{proof}
	Assume that (I) holds. To verify (II), we consider Proposition 3.3.2 and Corollary 3.3.3 from \cite{BrionHandbook}. These results reduce us to proving that $\overline{G}$ is $G_L$-quasi-projective, i.e. that there exist a finite-dimensional $G_L$-module $V$ and a $G_L$-equivariant locally closed immersion $\overline{G}\longrightarrow\mathbb{P}(V)$. We first observe that the Pl\"ucker embedding
	$$\mathrm{Gr}(n,\g\oplus\g)\longrightarrow\mathbb{P}(\Lambda^n(\g\oplus\g))$$ is ($G\times G$)-equivariant. Its restriction to $\overline{G}$ is therefore a ($G\times G$)-equivariant closed immersion
	$$\overline{G}\longrightarrow \mathbb{P}(\Lambda^n(\g\oplus\g)).$$ It follows that $\overline{G}$ is indeed $G_L$-quasi-projective, verifying (II).
	
	Now assume that (II) is true, and note that $$\overline{\mu}_L^{-1}(0)=\bigg\{(x,(\gamma,(y_1,y_2)))\in X\times (\overline{G}\times(\mathfrak{g}\oplus\mathfrak{g})):y_1=\nu(x)\text{ and }(\nu(x),y_2)\in\gamma\bigg\}.$$ One readily deduces that
	$$\overline{\mu}_L^{-1}(0)\longrightarrow X\times\overline{G}\times\mathfrak{g},\quad (x,(\gamma,(y_1,y_2)))\longrightarrow (x,\gamma,y_2)$$ is a $G_L$-equivariant closed immersion, where $X\times\overline{G}\times\mathfrak{g}$ is regarded as the product of the $G_L$-variety $X\times\overline{G}$ and the $G_L$-variety $\g$ with trivial action. We also know that $X\times\overline{G}\times\g$ has a geometric quotient by $G_L$, as (II) is assumed to be true. These last two sentences combine with Proposition \ref{Proposition: Restriction of quotient} and imply that $\overline{\mu}_L^{-1}(0)$ has a geometric quotient by $G_L$, i.e. (III) is true.
	
	Now assume that (III) holds, and consider the geometric quotient map
	$$\overline{\pi}_L:\overline{\mu}_L^{-1}(0)\longrightarrow \overline{X}.$$ Let us also observe that the $G$-action on $X\times (\overline{G\times\mathcal{S}})$ comes from restricting the $G_L$-action on $X\times T^*\overline{G}(\log D)$. Proposition \ref{Proposition: Restriction of quotient} then implies that $\overline{\pi}_L$ restricts to a geometric quotient
	$$(X\times(\overline{G\times\mathcal{S}}))\cap\overline{\mu}_L^{-1}(0)\longrightarrow\overline{\pi}_L((X\times(\overline{G\times\mathcal{S}}))\cap\overline{\mu}_L^{-1}(0))$$ of $(X\times(\overline{G\times\mathcal{S}}))\cap\overline{\mu}_L^{-1}(0)$ by $G$. On the other hand, $(X\times(\overline{G\times\mathcal{S}}))\cap\overline{\mu}_L^{-1}(0)$ is precisely the fibre of $\overline{\mu}_{\tau}:X\times(\overline{G\times\mathcal{S}})\longrightarrow\g$ over $0$. These last two sentences imply that $\overline{\mu}_{\tau}^{-1}(0)$ has a geometric quotient by $G$, i.e. (IV) holds. 
\end{proof}

This lemma turns out to yield a large class of Hamiltonian $G$-varieties $X$ such that $\overline{X}_{\tau}$ exists for all $\mathfrak{sl}_2$-triples $\tau$. To obtain this class, let $Y$ be a smooth $G$-variety. The $G$-action on $Y$ has a canonical lift to a Hamiltonian $G$-action on $T^*Y$, and there is a canonical moment map. This leads to the following result.

\begin{cor}\label{Corollary: Useful}
Let $Y$ be an irreducible smooth principal $G$-bundle and set $X=T^*Y$. Suppose that $\tau$ is an $\mathfrak{sl}_2$-triple in $\g$. The following statements hold.
\begin{itemize}
\item[(i)] The variety $\overline{X}_{\tau}$ exists and is a smooth Poisson variety.
\item[(ii)] If $\overline{X}_{\tau}$ is irreducible, then it is log symplectic and $k_{\tau}:X_{\tau}\longrightarrow\overline{X}_{\tau}$ is a symplectomorphism onto the unique open dense symplectic leaf in $\overline{X}_{\tau}$.   
\end{itemize}
\end{cor}
\begin{proof}
We begin by establishing that $\overline{X}_{\tau}$ exists. By Lemmas \ref{Lemma: FreeActionGbundle} and \ref{Lemma: Implications}, it suffices to show that $G$-action on $X$ admits a good quotient. Since $G$ acts freely on $Y$, this is implied by \cite[Theorem 2.3]{DrezetNarasimhan} and the remark that follows it.

Now observe that the $G$-action on $X=T^*Y$ is also free, implying that $\overline{X}_{\tau}=\overline{X}_{\tau}^{\circ}$. Proposition \ref{Proposition: Poisson structure} then tells us that $\overline{X}_{\tau}$ is smooth and Poisson. This completes the proof of (i). The assertion (ii) follows immediately from Theorem \ref{Theorem: nice}.
\end{proof}

\subsection{The main examples}
We now discuss some of the examples that motivate and best exhibit the results in this paper. Our first two examples satisfy the hypotheses of Corollary \ref{Corollary: Useful}, while the third example has a different nature. 

\begin{ex}
Suppose that $Y=G$ is endowed with the $G$-action defined by
$$g\cdot h:=hg^{-1},\quad g,h\in G.$$ The induced Hamiltonian $G$-action on $X=T^*Y=T^*G$ then satisfies
$$X_{\tau}\cong G\times\mathcal{S}_{\tau}\quad\text{and}\quad\overline{X}_{\tau}=(T^*G\times(\overline{G\times\mathcal{S}_{\tau}}))\sslash G\cong\overline{G\times\mathcal{S}_{\tau}}$$
for any $\mathfrak{sl}_2$-triple $\tau$ in $\g$. The fibrewise compactification in Theorem \ref{Theorem: Compact} becomes the one mentioned in Remark \ref{Remark: Projective fibres}.
\end{ex}

\begin{ex}
A mild generalization of the previous example can be obtained as follows. Suppose that $G$ is a closed subgroup of a connected linear algebraic group $Y=H$ with Lie algebra $\h$. Note that $G$ then acts on $Y$ via by the formula
$$g\cdot h:=hg^{-1},\quad g\in G,\text{ }h\in H.$$ The cotangent bundle $X=T^*Y=T^*H$ is thereby a Hamiltonian $G$-variety, and the left trivialization gives an identification $X\cong H\times\mathfrak{h}^*$. The moment map is given by 
$$H\times\h^*\to \g^*,\qquad (h,\alpha)\mapsto -\alpha\vert_{\g}.$$ 
One finds that $X_\tau \cong H\times((\mathcal S_\tau) \times \g^\dagger)$ under appropriate identifications, where $\g^{\dagger}$ denotes the annihilator of $\g$ in $\h^*$. By Corollary \ref{Corollary: Useful}, $\overline{X}_\tau$ exists for all $\mathfrak{sl}_2$-triples $\tau$ in $\g$.
\end{ex}

\begin{ex}
Let $\tau$ be a principal $\mathfrak{sl}_2$-triple in $\g$ and recall the notation used in Subsection \ref{Subsection: KW}. Consider the Hamiltonian $G$-varieties $X=G\times\mathcal{S}_{\tau}$ and $\overline{G\times\mathcal{S}_{\tau}}$, as well as the moment maps
$$\rho_{\tau}:G\times\mathcal{S}_{\tau}\longrightarrow\g\quad\text{and}\quad\overline{\rho}_{\tau}:\overline{G\times\mathcal{S}_{\tau}}\longrightarrow\g.$$ The discussion of $\mathcal{Z}_{\g}^{\tau}$ and $\overline{\mathcal{Z}_{\g}^{\tau}}$ in Subsection \ref{Subsection: Relation} combines with Proposition \ref{Proposition: PoissonSlice} to imply that
$$X_{\tau}=\rho_{\tau}^{-1}(\mathcal{S}_{\tau})=\mathcal{Z}_{\g}^{\tau}\quad\text{and}\quad\overline{X}_{\tau}=((G\times\mathcal{S}_{\tau})\times(\overline{G\times\mathcal{S}_{\tau}}))\sslash G\cong\overline{\rho}_{\tau}^{-1}(\mathcal{S}_{\tau})=\overline{\mathcal{Z}_{\g}^{\tau}}.$$ The fibrewise compactification in Theorem \ref{Theorem: Compact} becomes B\u{a}libanu's fibrewise compactification \eqref{Equation: Non-equivariant diagram}.
\end{ex}

\section*{Notation}
\begin{itemize}
\item $\mathcal{O}_Y$ --- structure sheaf of an algebraic variety $Y$
\item $\mathbb{C}[Y]$ --- coordinate ring of an algebraic variety $Y$
\item $G$ --- complex semisimple linear algebraic group
\item $G_L$ --- the subgroup $G\times\{e\}\subseteq G\times G$
\item $G_R$ --- the subgroup $\{e\}\times G\subseteq G\times G$
\item $\g$ --- Lie algebra of $G$
\item $\mathrm{Ad}:G\longrightarrow\operatorname{GL}(\g)$ --- adjoint representation
\item $\g_\Delta$ --- diagonal in $\g\oplus\g$
\item $n$ --- dimension of $\g$
\item $\langle\cdot,\cdot\rangle$ --- Killing form on $\g$
\item $\tau$ --- $\mathfrak{sl}_2$-triple in $\g$
\item $\mathcal S_\tau$ --- Slodowy slice associated to $\tau$.
\item $\chi:\g\longrightarrow\mathrm{Spec}(\mathbb{C}[\g]^G)$ --- adjoint quotient 
\item $y_{\tau}$ --- unique point at which $\mathcal{S}_{\tau}$ meets $\chi^{-1}(\chi(y))$, if $\tau$ is a principal $\mathfrak{sl}_2$-triple
\item $\rho=(\rho_L,\rho_R):T^*G\longrightarrow\g\oplus\g$ --- moment map for the ($G\times G$)-action on $T^*G$
\item $\rho_{\tau}:G\times\mathcal{S}_{\tau}\longrightarrow\g$ --- moment map for the $G$-action on $G\times\mathcal{S}_{\tau}$
\item $X$ --- Hamiltonian $G$-variety
\item $\nu:X\longrightarrow\g$ --- moment map for the $G$-action on $X$
\item $X_{\tau}$ --- the Poisson slice $\nu^{-1}(\mathcal{S}_{\tau})$
\item $X/G$ --- geometric quotient of $X$ by $G$
\item $\mu=(\mu_L,\mu_R):X\times T^*G\longrightarrow\g\oplus\g$ --- moment map for the ($G\times G$)-action on $X\times T^*G$
\item $\mu_{\tau}:X\times(G\times\mathcal{S}_{\tau})\longrightarrow\g$ moment map for the $G$-action on $X\times(G\times\mathcal{S}_{\tau})$
\item $\psi_{\tau}:X_{\tau}\longrightarrow (X\times (G\times\mathcal{S}_{\tau}))\sslash G$ --- canonical Poisson variety isomorphism
\item $\overline{G}$ --- De Concini--Procesi wonderful compactification of $G$
\item $D$ --- the divisor $\overline{G}\setminus G$
\item $T^*\overline{G}(\log(D))$ --- log cotangent bundle of $(\overline{G},D)$
\item $\overline{\rho}=(\overline{\rho}_L,\overline{\rho}_R):T^*\overline{G}(\log(D))\longrightarrow\g\oplus\g$ --- moment map for the ($G\times G$)-action on $T^*\overline{G}(\log(D))$
\item $\overline{G\times\mathcal{S}_{\tau}}$ --- the Poisson slice $\overline{\rho}_R^{-1}(\mathcal{S}_{\tau})$
\item $\mathcal{Z}_{\mathfrak{g}}^{\tau}$ --- universal centralizer of $\g$
\item $\overline{\mathcal{Z}_{\g}^{\tau}}$ --- B\u{a}libanu's partial compactification of $\mathcal{Z}_{\g}^{\tau}$
\item $\overline{\rho}_{\tau}:\overline{G\times\mathcal{S}_{\tau}}\longrightarrow\g$ --- moment map for the $G$-action on $\overline{G\times\mathcal{S}_{\tau}}$
\item $\overline{\mu}=(\overline{\mu}_L,\overline{\mu}_R):X\times T^*\overline{G}(\log D)\longrightarrow\g\oplus\g$ --- moment map for the ($G\times G$)-action on $X\times T^*\overline{G}(\log D)$
\item $\overline{\mu}_{\tau}:X\times (\overline{G\times\mathcal{S}_{\tau}})\longrightarrow\g$ --- moment map for the $G$-action on $X\times (\overline{G\times\mathcal{S}_{\tau}})$
\item $\overline{X}$ --- the Hamiltonian reduction $(X\times T^*\overline{G}(\log D))\sslash G_L$
\item $k:X\longrightarrow\overline{X}$ --- canonical $G$-equivariant open embedding
\item $\overline{\nu}:\overline{X}\longrightarrow\g$ --- equivariant extension of $\nu$ to $\overline{X}$
\item $\overline{X}_{\tau}$ --- the Hamiltonian reduction $(X\times (\overline{G\times\mathcal{S}_{\tau}}))\sslash G$
\item $k_{\tau}:X_{\tau}\longrightarrow\overline{X}_{\tau}$ --- canonical open embedding
\item $(X\times(\overline{G\times\mathcal{S}_{\tau}}))^{\circ}$ --- set of points in $X\times (\overline{G\times\mathcal{S}_{\tau}})$ with trivial $G$-stabilizers
\item $\overline{X}_{\tau}^{\circ}$ --- the Hamiltonian reduction $(X\times(\overline{G\times\mathcal{S}_{\tau}}))^{\circ}\sslash G$ 
\end{itemize}

\bibliographystyle{acm} 
\bibliography{Extension}

\begin{thebibliography}{10}

\bibitem{AbeCrooks}
{\sc Abe, H., and Crooks, P.}
\newblock Hessenberg varieties, {S}lodowy slices, and integrable systems.
\newblock {\em Math. Z. 291}, 3--4 (2019), 1093--1132.

\bibitem{BalibanuUniversal}
{\sc B{\u a}libanu, A.}
\newblock The partial compactification of the universal centralizer.
  ar{X}iv:1710.06327 (2019), 21pp.

\bibitem{Bezrukavnikov}
{\sc Bezrukavnikov, R., and Finkelberg, M.}
\newblock Equivariant {S}atake category and {K}ostant--{W}hittaker reduction.
\newblock {\em Mosc. Math. J. 8}, 1 (2008), 39--72, 183.

\bibitem{BFM}
{\sc Bezrukavnikov, R., Finkelberg, M., and Mirkovi\'{c}, I.}
\newblock Equivariant homology and {$K$}-theory of affine {G}rassmannians and
  {T}oda lattices.
\newblock {\em Compos. Math. 141}, 3 (2005), 746--768.

\bibitem{Bielawski}
{\sc Bielawski, R.}
\newblock Hyperk{\"a}hler structures and group actions.
\newblock {\em J. London Math. Soc. (2) 55}, 2 (1997), 400--414.

\bibitem{BielawskiComplex}
{\sc Bielawski, R.}
\newblock Slices to sums of adjoint orbits, the {A}tiyah-{H}itchin manifold,
  and {H}ilbert schemes of points.
\newblock {\em Complex Manifolds 4}, 1 (2017), 16--36.

\bibitem{BFN}
{\sc Braverman, A., Finkelberg, M., and Nakajima, H.}
\newblock Towards a mathematical definition of {C}oulomb branches of
  3-dimensional {$N=4$} gauge theories, {II}.
\newblock {\em Adv. Theor. Math. Phys. 22}, 5 (2018), 1071--1147.

\bibitem{BrionHandbook}
{\sc Brion, M.}
\newblock Linearization of algebraic group actions.
\newblock In {\em Handbook of group actions. {V}ol. {IV}}, vol.~41 of {\em Adv.
  Lect. Math. (ALM)}. Int. Press, Somerville, MA, 2018, pp.~291--340.

\bibitem{Weinstein}
{\sc Cannas~da Silva, A., and Weinstein, A.}
\newblock {\em Geometric models for noncommutative algebras}, vol.~10 of {\em
  Berkeley Mathematics Lecture Notes}.
\newblock American Mathematical Society, Providence, RI; Berkeley Center for
  Pure and Applied Mathematics, Berkeley, CA, 1999.

\bibitem{Cavalcanti}
{\sc Cavalcanti, G.~R., and Klaasse, R.~L.}
\newblock Fibrations and log-symplectic structures.
\newblock {\em J. Symplectic Geom. 17}, 3 (2019), 603--638.

\bibitem{CrooksBulletin}
{\sc Crooks, P.}
\newblock An equivariant description of certain holomorphic symplectic
  varieties.
\newblock {\em Bull. Aust. Math. Soc. 97}, 2 (2018), 207--214.

\bibitem{CrooksJGP}
{\sc Crooks, P.}
\newblock Kostant-{T}oda lattices and the universal centralizer.
\newblock {\em J. Geom. Phys. 150\/} (2020), 103595, 16pp.

\bibitem{CrooksRayan}
{\sc Crooks, P., and Rayan, S.}
\newblock Abstract integrable systems on hyperk{\"a}hler manifolds arising from
  {S}lodowy slices.
\newblock {\em Math. Res. Lett. 26}, 9 (2019), 9--33.

\bibitem{CrooksRoeser2}
{\sc Crooks, P., and R{\"o}ser, M.}
\newblock Hessenberg varieties and {P}oisson slices. ar{X}iv:2005.00874 (2020),
  37pp.

\bibitem{CrooksVP}
{\sc Crooks, P., and van Pruijssen, M.}
\newblock An application of spherical geometry to hyperk\"ahler slices. {T}o
  appear in \textit{{C}anad. {J}. {M}ath}. (2020),
  doi:10.4153/s0008414x20000127, 30pp.

\bibitem{DancerKirwanRoser}
{\sc Dancer, A., Kirwan, F., and R\"{o}ser, M.}
\newblock Hyperk\"{a}hler implosion and {N}ahm's equations.
\newblock {\em Comm. Math. Phys. 342}, 1 (2016), 251--301.

\bibitem{DancerKirwan}
{\sc Dancer, A., Kirwan, F., and Swann, A.}
\newblock Implosion for hyperk\"{a}hler manifolds.
\newblock {\em Compos. Math. 149}, 9 (2013), 1592--1630.

\bibitem{DeConcini}
{\sc De~Concini, C., and Procesi, C.}
\newblock Complete symmetric varieties.
\newblock In {\em Invariant theory ({M}ontecatini, 1982)}, vol.~996 of {\em
  Lecture Notes in Math.} Springer, Berlin, 1983, pp.~1--44.

\bibitem{Dimofte}
{\sc Dimofte, T., and Garner, N.}
\newblock Coulomb branches of star-shaped quivers.
\newblock {\em J. High Energy Phys.}, 2 (2019), 004, front matter+87.

\bibitem{DrezetNarasimhan}
{\sc Drezet, J.-M., and Narasimhan, M.~S.}
\newblock Groupe de {P}icard des vari\'{e}t\'{e}s de modules de fibr\'{e}s
  semi-stables sur les courbes alg\'{e}briques.
\newblock {\em Invent. Math. 97}, 1 (1989), 53--94.

\bibitem{Frejlich}
{\sc Frejlich, P., and M\u{a}rcu\c{t}, I.}
\newblock The normal form theorem around {P}oisson transversals.
\newblock {\em Pacific J. Math. 287}, 2 (2017), 371--391.

\bibitem{Gan}
{\sc Gan, W.~L., and Ginzburg, V.}
\newblock Quantization of {S}lodowy slices.
\newblock {\em Int. Math. Res. Not.}, 5 (2002), 243--255.

\bibitem{Goto}
{\sc Goto, R.}
\newblock Rozansky-{W}itten invariants of log symplectic manifolds.
\newblock In {\em Integrable systems, topology, and physics ({T}okyo, 2000)},
  vol.~309 of {\em Contemp. Math.} Amer. Math. Soc., Providence, RI, 2002,
  pp.~69--84.

\bibitem{GualtieriLi}
{\sc Gualtieri, M., and Li, S.}
\newblock Symplectic groupoids of log symplectic manifolds.
\newblock {\em Int. Math. Res. Not. IMRN}, 11 (2014), 3022--3074.

\bibitem{Gualtieri}
{\sc Gualtieri, M., Li, S., Pelayo, {\'A}., and Ratiu, T.~S.}
\newblock The tropical momentum map: a classification of toric log symplectic
  manifolds.
\newblock {\em Math. Ann. 367}, 3-4 (2017), 1217--1258.

\bibitem{GualtieriPym}
{\sc Gualtieri, M., and Pym, B.}
\newblock Poisson modules and degeneracy loci.
\newblock {\em Proc. Lond. Math. Soc. (3) 107}, 3 (2013), 627--654.

\bibitem{GJS}
{\sc Guillemin, V., Jeffrey, L., and Sjamaar, R.}
\newblock Symplectic implosion.
\newblock {\em Transform. Groups 7}, 2 (2002), 155--184.

\bibitem{GuilleminLermanSternberg}
{\sc Guillemin, V., Lerman, E., and Sternberg, S.}
\newblock {\em Symplectic fibrations and multiplicity diagrams}.
\newblock Cambridge University Press, Cambridge, 1996.

\bibitem{GMP}
{\sc Guillemin, V., Miranda, E., and Pires, A.~R.}
\newblock Symplectic and {P}oisson geometry on {$b$}-manifolds.
\newblock {\em Adv. Math. 264\/} (2014), 864--896.

\bibitem{GuilleminIMRN}
{\sc Guillemin, V., Miranda, E., and Weitsman, J.}
\newblock Desingularizing {$b^m$}-symplectic structures.
\newblock {\em Int. Math. Res. Not. IMRN}, 10 (2019), 2981--2998.

\bibitem{GuilleminPhysics}
{\sc Guillemin, V., and Sternberg, S.}
\newblock {\em Symplectic techniques in physics}, second~ed.
\newblock Cambridge University Press, Cambridge, 1990.

\bibitem{GuilleminAdv}
{\sc Guillemin, V.~W., Miranda, E., and Weitsman, J.}
\newblock On geometric quantization of {$b$}-symplectic manifolds.
\newblock {\em Adv. Math. 331\/} (2018), 941--951.

\bibitem{Hartshorne}
{\sc Hartshorne, R.}
\newblock {\em Algebraic geometry}.
\newblock Springer-Verlag, New York-Heidelberg, 1977.
\newblock Graduate Texts in Mathematics, No. 52.

\bibitem{Hilgert}
{\sc Hilgert, J., Manon, C., and Martens, J.}
\newblock Contraction of {H}amiltonian {$K$}-spaces.
\newblock {\em Int. Math. Res. Not. IMRN}, 20 (2017), 6255--6309.

\bibitem{KostantLie}
{\sc Kostant, B.}
\newblock Lie group representations on polynomial rings.
\newblock {\em Amer. J. Math. 85\/} (1963), 327--404.

\bibitem{KostantWhittaker}
{\sc Kostant, B.}
\newblock On {W}hittaker vectors and representation theory.
\newblock {\em Invent. Math. 48}, 2 (1978), 101--184.

\bibitem{Kronheimer}
{\sc Kronheimer, P.~B.}
\newblock A hyper-{K}\"{a}hlerian structure on coadjoint orbits of a semisimple
  complex group.
\newblock {\em J. London Math. Soc. (2) 42}, 2 (1990), 193--208.

\bibitem{LermanMeinrenken}
{\sc Lerman, E., Meinrenken, E., Tolman, S., and Woodward, C.}
\newblock Nonabelian convexity by symplectic cuts.
\newblock {\em Topology 37}, 2 (1998), 245--259.

\bibitem{Marsden}
{\sc Marsden, J.~E., Ratiu, T., and Raugel, G.}
\newblock Symplectic connections and the linearisation of {H}amiltonian
  systems.
\newblock {\em Proc. Roy. Soc. Edinburgh Sect. A 117}, 3-4 (1991), 329--380.

\bibitem{Mayrand}
{\sc Mayrand, M.}
\newblock Stratified hyperk\"ahler spaces and {N}ahm's equations. {P}h{D}
  thesis, {U}niversity of {O}xford (2019), 162pp.

\bibitem{MooreTachikawa}
{\sc Moore, G.~W., and Tachikawa, Y.}
\newblock On 2d {TQFT}s whose values are holomorphic symplectic varieties.
\newblock In {\em String-{M}ath 2011}, vol.~85 of {\em Proc. Sympos. Pure
  Math.} Amer. Math. Soc., Providence, RI, 2012, pp.~191--207.

\bibitem{MarkusTorres2}
{\sc M\u{a}rcu\c{t}, I., and Osorno~Torres, B.}
\newblock Deformations of log-symplectic structures.
\newblock {\em J. Lond. Math. Soc. (2) 90}, 1 (2014), 197--212.

\bibitem{MarkusTorres1}
{\sc M\u{a}rcu\c{t}, I., and Osorno~Torres, B.}
\newblock On cohomological obstructions for the existence of log-symplectic
  structures.
\newblock {\em J. Symplectic Geom. 12}, 4 (2014), 863--866.

\bibitem{Mustata}
{\sc Musta\c{t}\v{a}, M.}
\newblock An irreducibility criterion.
  $http://www-personal.umich.edu/~mmustata/note1\_09.pdf$. (2009).

\bibitem{Patil}
{\sc Patil, D.~P., and Storch, U.}
\newblock {\em Introduction to algebraic geometry and commutative algebra},
  vol.~1 of {\em IISc Lecture Notes Series}.
\newblock IISc Press, Bangalore; World Scientific Publishing Co. Pte. Ltd.,
  Hackensack, NJ, 2010.

\bibitem{PremetSpecial}
{\sc Premet, A.}
\newblock Special transverse slices and their enveloping algebras.
\newblock {\em Adv. Math. 170}, 1 (2002), 1--55.
\newblock With an appendix by Serge Skryabin.

\bibitem{PymElliptic}
{\sc Pym, B.}
\newblock Elliptic singularities on log symplectic manifolds and
  {F}eigin-{O}desskii {P}oisson brackets.
\newblock {\em Compos. Math. 153}, 4 (2017), 717--744.

\bibitem{Pym}
{\sc Pym, B.}
\newblock Constructions and classifications of projective {P}oisson varieties.
\newblock {\em Lett. Math. Phys. 108}, 3 (2018), 573--632.

\bibitem{PymSchedler}
{\sc Pym, B., and Schedler, T.}
\newblock Holonomic {P}oisson manifolds and deformations of elliptic algebras.
\newblock In {\em Geometry and physics. {V}ol. {II}}. Oxford Univ. Press,
  Oxford, 2018, pp.~681--703.

\bibitem{Radko}
{\sc Radko, O.}
\newblock A classification of topologically stable {P}oisson structures on a
  compact oriented surface.
\newblock {\em J. Symplectic Geom. 1}, 3 (2002), 523--542.

\bibitem{Riche}
{\sc Riche, S.}
\newblock Kostant section, universal centralizer, and a modular derived
  {S}atake equivalence.
\newblock {\em Math. Z. 286}, 1-2 (2017), 223--261.

\bibitem{Schmitt}
{\sc Schmitt, A. H.~W.}
\newblock {\em Geometric invariant theory and decorated principal bundles}.
\newblock Zurich Lectures in Advanced Mathematics. European Mathematical
  Society (EMS), Z\"{u}rich, 2008.

\bibitem{Slodowy}
{\sc Slodowy, P.}
\newblock {\em Simple singularities and simple algebraic groups}, vol.~815 of
  {\em Lecture Notes in Mathematics}.
\newblock Springer, Berlin, 1980.

\bibitem{Tauvel}
{\sc Tauvel, P., and Yu, R. W.~T.}
\newblock {\em Lie algebras and algebraic groups}.
\newblock Springer Monographs in Mathematics. Springer-Verlag, Berlin, 2005.

\bibitem{TelemanICM}
{\sc Teleman, C.}
\newblock Gauge theory and mirror symmetry.
\newblock In {\em Proceedings of the {I}nternational {C}ongress of
  {M}athematicians---{S}eoul 2014. {V}ol. {II}\/} (2014), Kyung Moon Sa, Seoul,
  pp.~1309--1332.

\bibitem{Teleman}
{\sc Teleman, C.}
\newblock The role of {C}oulomb branches in 2{D} gauge theory.
  ar{X}iv:1801.10124 (2019), 18pp.

\end{thebibliography}
\end{document}